\documentclass[reqno,11pt]{amsart}
\usepackage{geometry}
\usepackage{mathtools,amssymb,amsthm,mathrsfs,color,lineno,paralist,graphicx,float}
\usepackage[colorlinks,
linkcolor=blue,
anchorcolor=green,
citecolor=blue,
]{hyperref}

\usepackage[T1]{fontenc} 
\usepackage[utf8]{inputenc}

\usepackage{calc}
\definecolor{bleu1}{RGB}{0,57,128}
\def\bleu1{\color{bleu1}}

\usepackage{etoolbox}
\patchcmd{\section}{\normalfont}{\normalfont \bleu1}{}{}
\patchcmd{\subsection}{\normalfont}{\normalfont \bleu1}{}{}
\patchcmd{\subsubsection}{\normalfont}{\normalfont \bleu1}{}{}

\date{}
\usepackage{enumitem}

\newtheorem{theorem}{Theorem}[section]
\newtheorem{corollary}{Corollary}[section]
\newtheorem{proposition}{Proposition}[section]
\newtheorem{lemma}{Lemma}[section]

\theoremstyle{definition}
\newtheorem{remark}{Remark}[section]
\newtheorem{claim}{Claim}

\newcommand{\dif}{\mathrm{d}}   

\numberwithin{equation}{section}

\newcommand{\N}{{\mathbb N}}
\newcommand{\Q}{{\mathbb Q}}
\newcommand{\R}{{\mathbb R}}
\newcommand{\C}{\mathbb{C}}
\newcommand{\T}{{\mathbb T}}

\newcommand{\Z}{{\mathbb Z}}
\newcommand{\ii}{{\mathrm{i}}} 

\DeclareMathOperator{\meas}{mes} 

\begin{document}
\pagestyle{plain}
\title{
Stability of Spectral Types of Quasi-Periodic Schrödinger Operators With Respect to Perturbations by Decaying Potentials}

\author{David Damanik}
\address{ Department of Mathematics, Rice University, Houston, Texas, 77005}
\email{damanik@rice.edu}

\author{Xianzhe Li}
\address{Chern Institute of Mathematics and LPMC, Nankai University, Tianjin 300071, China} 
\email{xianzheli@mail.nankai.edu.cn}

\author {Jiangong You}
\address{
Chern Institute of Mathematics and LPMC, Nankai University, Tianjin 300071, China} \email{jyou@nankai.edu.cn}

\author{Qi Zhou}
\address{
Chern Institute of Mathematics and LPMC, Nankai University, Tianjin 300071, China
}

\email{qizhou@nankai.edu.cn}

\begin{abstract}
We consider perturbations of quasi-periodic Schrödinger operators on the integer lattice with analytic sampling functions by decaying potentials and seek decay conditions under which various spectral properties are preserved. In the (almost) reducibility regime we prove that for perturbations with finite first moment, the essential spectrum remains purely absolutely continuous and the newly created discrete spectrum must be finite in each gap of the unperturbed spectrum. We also prove that for fixed phase, Anderson localization occurring for almost all frequencies in the regime of positive Lyapunov exponents is preserved under exponentially decaying perturbations.
\end{abstract}

\maketitle

\tableofcontents

\section{Introduction}

\subsection{Background}

Let $ H_0 $ denote the one dimensional Schr\"odinger operator acting on $ \ell^2(\Z) $,
\begin{equation}\label{H_0}
    H_0=\Delta+v_0(n)\delta_{nn'},
\end{equation} 
where $ \Delta $ denotes the discrete Laplacian
$$
    \Delta(n,n')=
    \begin{cases}
        1&\text{ if }|n-n'|=1,\\
        0&\text{ otherwise}.
    \end{cases}
$$ 
and $ v_0 $ is assumed to be real and bounded, so that $ H_0 $ is bounded and self-adjoint, and its spectrum, $ \sigma(H_0) $, is a compact subset of $ \R $. 

In physics, the operator \eqref{H_0} describes a charged particle, such as an electron, in the electric field $ v_0 $.  It is of interest to determine which of the spectral properties of $ H_0 $ are preserved if the operator is perturbed by a suitably decaying potential.  Problems of this nature have been extensively studied dating back to 1910s.   We give a short review and refer the curious reader to \cite{denisov2007spectral}. 

Consider the following perturbed one-dimensional Schr\"odinger operators acting on $ \ell^2(\Z) $ or $ \ell^2(\N) $ (with Dirichlet boundary condition),
\begin{equation} 
    H_g=H_0+g(n)\delta_{nn'},
\end{equation} 
where $ g $ is a real perturbation decaying at infinity. We say that $g : \Z \to \C$ is \emph{decaying} if
$$
\lim_{|n| \to \infty} g(n) = 0
$$ 
and denote the space of all decaying $g : \Z \to \R$ by $c_0$. For the discussion below, it is also convenient to consider
\begin{align*}
\ell^{k,p} & = \Big\{ g : \Z \to \R : \left(n^k g(n)\right)_{n \in \Z} \in \ell^p \Big\}, \quad k > 0, \; p \ge 1, \\
P(\gamma) & = \Big\{ g : \Z \to \R : \left(n^\gamma g(n)\right)_{n \in \Z} \in \ell^\infty \Big\}, \quad \gamma > 0, \\
p(\gamma) & = \Big\{ g : \Z \to \R : \left(n^\gamma g(n)\right)_{n \in \Z} \in c_0 \Big\}, \quad \gamma > 0
\end{align*}
to capture weighted $\ell^p$ and power-law decay.

There are two main stability aspects. One is the preservation of the essential spectrum, which holds in complete generality. That is, the essential spectra of $ H_0 $ and $ H_g $ coincide for all $g \in c_0$ by the classical Weyl criterion. In fact, in the case $v_0 = 0$, the converse is known to hold as well, that is, the addition of any (real) $g \in \ell^\infty \setminus c_0$ generates essential spectrum outside $\sigma_\mathrm{ess}(H_0) = [-2,2]$ \cite{DHKS03}.

The other is the preservation of the absolutely continuous spectrum. The (essential support of the) absolutely continuous spectrum of the original operator $ H_0 $ is preserved under $ \ell^1 $ perturbations by the Birman-Rosenblum-Kato theory (cf.\ Chapter 10 in \cite{kato2013perturbation}, Chapter~$\mathrm{\uppercase\expandafter{\romannumeral11}}$ in \cite{reed1978iv}). However, this does not include the statement that the spectral type remains purely absolutely continuous if the unperturbed operator has purely absolutely continuous spectrum, it is feasible that the common essential spectrum may support some singular spectrum, even for $\ell^1$ perturbations. In the appendix we make this explicit for $v_0 = 0$ and a suitable $g \in \ell^1$, based on an observation which we learned from Milivoje Lukic. It can happen that the boundary of the essential spectrum contains an eigenvalue. The example we give belongs to $P(2)$, and this is actually optimal on the power scale since it is also known that $g \in \ell^{1,1}$ ensures the absence of singular spectrum in $[-2,2]$; compare, for example, \cite[Section~7.5]{Tes1999Jacobi}. On the other hand, if $v_0 = 0$ and $g \in \ell^1$, it follows quickly via Pr\"ufer variables that the interior of the essential spectrum remains purely absolutely continuous since for those energies, the Pr\"ufer radius must remain bounded. Indeed, when $v_0 = 0$, the absence of singular spectrum in the interior of the essential spectrum will be true under the somewhat weaker assumption $g \in p(1)$, as shown by Remling \cite{remling1998absolutely}.

If $g \notin p(1)$, the absence of singular spectrum in the interior of the essential spectrum may also fail, as demonstrated by the celebrated Wigner-von Neumann example \cite{von1929merkwurdige,reed1978iv}, which perturbs $v_0 = 0$ with a $g \in P(1)$ and admits an embedded eigenvalue. Any slower decay may allow for embedded singular continuous spectrum, as shown by Kiselev \cite{K05}. On the other hand, Deift and Killip \cite{deift1999absolutely, killip2002perturbations} have shown (see also Christ-Kiselev-Remling \cite{christ1998absolutely, remling1998absolutely, christ1997absolutely} and Kiselev-Last-Simon \cite{kiselev1998modified} for related work) that if $ v_0=0 $ or $ v_0 $ periodic, the addition of $ g\in \ell^2 $ will preserve the essential support of the absolutely continuous spectrum, and this result is optimal in the sense of $ \ell^p $-type decay. It is conjectured that this stability result holds for any $ v_0 $. Further important related developments, which often are referred to as Killip-Simon-type problems or theorems, are contained in \cite{DKS10, KS03, KS09, Y18}.
Almost all the results and conjectures above (Yuditskii's result for finite-gap Jacobi matrices \cite{Y18} being a notable exception) have continuum analogs, and these results apply to both the half-line and whole-line settings. 

On the other hand, both point spectrum and singular continuous spectrum are in general sensitive to decaying perturbations.  Although it is known that for the Anderson model, Anderson localization (pure point spectrum with exponentially decaying eigenfunctions) holds on an interval with probability one (cf.~\cite{CKM87, KS80}), this is highly unstable and can be destroyed and turned into purely singular continuous spectrum by generic rank one perturbations (cf.~\cite{del1994operators, gordon1994pure, simon1995spectral}), which then provide examples that are arbitrarily small and arbitrarily fast decaying.

\subsection{Decaying Perturbations of Quasi-Periodic Potentials}

In this paper, we investigate the stability problem for perturbations of quasi-periodic Schr\"odinger operators by decaying potentials. 

Let  $H_{v,\alpha,\theta}$ be a quasi-periodic discrete Schr\"odinger operator acting on $ \ell^2(\Z) $ defined as follows,
\begin{equation}\label{qpH}
    H_{v,\alpha,\theta}=\Delta+v(n\alpha+\theta)\delta_{n n'},
\end{equation}   
where $ v $ is real analytic on $ \T^d $ and $ \alpha,\theta\in\R^d $ are parameters called the \textit{frequency} and \textit{phase}, respectively. This operator, especially the famous almost Mathieu operator (AMO), which arises when $d = 1$, $v(\cdot)=2\lambda\cos2\pi(\cdot)$, and $ \alpha\in\R\backslash\Q $,
has received a lot of attention because of its relevance in physics and also because it provides fruitful examples in spectral theory. It is known that the spectrum and the spectral type of \eqref{qpH} are deeply influenced by the arithmetic properties of $ \alpha, \theta $ and the largeness of $ v $, that spectral types may coexist (cf. \cite{bjerklov2019coexistence}), and that exact mobility edges can occur (c.f.\cite{wang2021exact}). One may consult \cite{avila2008absolutely,avila2009ten, AJ2010Almost, avila2017sharp,bourgain2000nonperturbative,eliasson,ge2019exponential,jitomirskaya1999metal,jitomirskaya2018universal,JL2018Universal} and references therein for further information.

We are interested in the spectral type of perturbations of quasi-periodic Schr\"odinger operators by decaying potentials, that is, we consider the following discrete Schr\"odinger operator on $ \ell^2(\Z) $, 
\begin{equation} \label{pH}
    \widetilde{H}_{v,\alpha,\theta}=H_{v,\alpha,\theta}+g(n)\delta_{n n'}
\end{equation} 
with a decaying $ g $. The results we obtain can naturally be grouped according to whether $d = 1$ (the \emph{one-frequency case}) or $d \ge 2$ (the \emph{multi-frequency case}), as the same distinction applies in the unperturbed case.

\subsubsection{Local Results in the Multi-Frequency Case}

We begin with a discussion of results that hold for all $d \ge 1$, and hence apply in particular in the multi-frequency case. For the sake of simplicity, in this subsection, we fix $ v_0\in C^\omega(\T^d,\R) $, and replace $ v $ by $ \lambda v_0 $, where $ \lambda\in\R $ is called the \textit{coupling constant}. We denote the operators \eqref{qpH} and \eqref{pH} by $ H_{\lambda,\alpha,\theta} $ and $ \widetilde{H}_{\lambda,\alpha,\theta} $, respectively. The spectrum and the essential spectrum of an operator $ H $ are denoted by  $ \sigma(H) $ and $ \sigma_\mathrm{ess}(H) $, respectively.  It is known that for any $ (1,\alpha) $ rationally independent, $ \sigma(H_{\lambda,\alpha,\theta})=\sigma_\mathrm{ess}(H_{\lambda,\alpha,\theta})=\Sigma_{\lambda,\alpha} $, does not depend on $ \theta $, and by Weyl's criterion, $ \sigma_\mathrm{ess}(\widetilde{H}_{\lambda,\alpha,\theta})=\sigma_\mathrm{ess}(H_{\lambda,\alpha,\theta}) $. In the following, we will always assume  $\alpha$ to be Diophantine. Recall that $ \alpha $ is said to be \emph{Diophantine} if there exist $ \gamma>0 $ and $ \tau>d-1  $ such that $ \alpha\in \mathrm{DC}_d(\gamma,\tau) $, where
\[
    \mathrm{DC}_d(\gamma,\tau):=\left\{x\in\R^d:\inf_{j\in\Z}|\langle n,x\rangle- j|\geq\frac{\gamma}{|n|^\tau}, 0\neq n\in\mathbb{Z}^d\right\}.
\]
The set $ \mathrm{DC}_d:=\bigcup_{\gamma>0,\tau>d-1}\mathrm{DC}_d(\gamma,\tau) $ has full Lebesgue measure.


\medskip
\noindent\textbf{Preservation of Purely Absolutely Continuous Spectrum.} When $ \lambda $ is sufficiently small and $\alpha$ is Diophantine, the spectrum of $ H_{\lambda,\alpha,\theta} $ is purely absolutely continuous \cite{avila2008absolutely, AJ2010Almost,eliasson}. Our stability result for the operator \eqref{pH} in the small coupling regime now reads as follows:

\begin{theorem}\label{acthm}
    Suppose that $ \alpha\in \mathrm{DC}_d$ and $ v_0\in C^\omega(\T^d,\R) $. Then there exists $ \lambda_0=\lambda_0(\alpha,d,v_0)>0 $ such that for every $ \lambda\in (0,\lambda_0) $ and $g \in \ell^{1,1}$, we have the following:
\begin{enumerate}
    \item  The restriction of the canonical spectral measure of $ \widetilde{H}_{\lambda,\alpha,\theta} $ to $ \Sigma=\sigma_\mathrm{ess}(\widetilde{H}_{\lambda,\alpha,\theta}) $ is purely absolutely continuous for every $ \theta\in\T^d $. 
    \item There are at most finitely many eigenvalues in each spectral gap of the unperturbed operator.
\end{enumerate}
\end{theorem}

\begin{remark}
    For $ d=1 $, we can choose $ \lambda_0 $ in a nonperturbative way, that is, uniformly in $ \gamma,\tau $. 
\end{remark}

\begin{remark}
The condition $g \in \ell^{1,1}$ is same as in the periodic case {\rm (}cf. \cite{GS1993Short, marchenko2011sturm, RKM2005Spectral, Tes1996Oscillation}{\rm )} and it is essentially optimal as the example given in the appendix (which belongs to $P(2)$) shows.
\end{remark}

Our result can be viewed as the stability of purely absolutely continuous spectral measures on the unperturbed spectrum. A similar setting was also considered in the continuum case by Kr\"uger in the paper \cite{kruger2010perturbations}, which mainly focused on the complement of the unperturbed spectrum.
    

\medskip
\noindent\textbf{Preservation of Anderson Localization.} When $ \lambda $ is sufficiently large, then the Lyapunov exponent is positive in the spectrum \cite{bourgain2005positivity,bourgain2000nonperturbative,SS1991Positive}, 
and the spectrum of \eqref{qpH} is typically pure point \cite{chulaevsky1993,eliasson1997discrete}. Furthermore, it displays \textit{Anderson localization} (AL), that is, pure point spectrum with exponentially decaying eigenfunctions. Anderson localization is an interesting phenomenon in the research of spectral theory and has been widely studied because of its importance in solid state physics.

We give a brief review here. When $ v_0$ is fixed, the operators \eqref{qpH} are a family of operators parameterized by $ (\theta,\alpha) $, thus there are mainly two directions to prove AL. For $ d=1 $, Fr\"ohlich–Spencer–Wittwer \cite{frohlich1990localization} and Sinai \cite{sinai1987anderson} proved that for \textit{fixed} Diophantine $ \alpha $, if $ v_0 $ is cosine-like, then \eqref{qpH} has AL for a.e. $ \theta $ and large enough $ \lambda $. Fix $ v_0(x)=2\cos 2\pi(x) $ specially, for any \textit{fixed} Diophantine $ \alpha $, Jitomirskaya \cite{jitomirskaya1999metal} showed that AL holds for $ \lambda>1 $ except a Lebesgue zero $ \theta $-set $ \Theta $, which has a concrete arithmetic description. This work is generalized to Liouvillean $ \alpha $ in the measure setting by Avila-You-Zhou \cite{avila2017sharp}, and in the arithmetic setting   by Jitomirskaya-Liu \cite{jitomirskaya2018universal}. On the other hand, for \textit{fixed} $ \theta $, Bourgain-Goldstein \cite{bourgain2000nonperturbative} proved that \eqref{qpH} has AL in the positive Lyapunov exponent regime, and a.e. Diophantine $ \alpha $. This work is extended to $ d $ arbitrary \cite{bourgain2005positivity}, and higher space dimensions \cite{bourgain2002anderson,bourgain2007derson}, by Bourgain, Goldstein and Schlag. Recently, the results in \cite{bourgain2007derson} have been largely extended by Jitomirskaya-Liu-Shi \cite{jitomirskaya2020anderson}.

Let us state our stability result concerning Anderson localization for the operator \eqref{pH} in the positive Lyapunov exponent regime:

\begin{theorem}\label{mainresult}
Assume that $v\in C^\omega(\T^d,\R)$ and the Lyapunov exponent of $ H_{v,\alpha,\theta} $ satisfies 
    $$
        L(\alpha,E)>c_0>0.
    $$
for all $ \alpha\in DC_d $ and $E\in (E_1,E_2)$. Fix $ \theta_0\in\T^d $. Then, for any $ g $ decaying exponentially, $ \widetilde{H}_{v,\alpha,\theta_0} $ exhibits Anderson localization in  $ (E_1,E_2) $ for almost all $ \alpha\in DC_d $.
\end{theorem}

%

\begin{remark}
Our result shows that for a fixed perturbation, most operators in the quasi-periodic family continue to have pure point spectrum. This should be contrasted with what rank-one perturbation theory (cf., e.g.,  \cite{del1994operators,gordon1994pure}) provides, namely for a \textit{fixed} operator and a family of perturbations, a generic but zero measure set of perturbations may destroy the pure point spectrum (if it occurs on some interval) and turn it into singular continuous spectrum. Thus, these are two different and complementary perspectives and phenomena. 
\end{remark}

\subsubsection{Global Picture in the One-Frequency Case}

Let us now turn our attention to the case $d = 1$. The spectral analysis of the one-frequency case is more complete, and indeed more ``global'', thanks in large part to Avila's global theory for one-frequency analytic quasi-periodic Schr\"odinger  cocycles \cite{avila2015global}. Indeed, Avila's global theory \cite{avila2015global} shows that all analytic $ \mathrm{SL}(2,\C) $ cocycles $ (\alpha,A) $ that are not \textit{uniformly hyperbolic} can be classified into three regimes: 
\begin{enumerate}
    \item \textit{Subcritical}, if there exists $ \delta> 0 $ such that the holomorphic extension of $ (\alpha,A) $, $ (\alpha,A_z) $, with $ L(\alpha,A_z)=0 $ through some strip $ |\Im z| \leq\delta $,
    \item \textit{Supercritical}, or \textit{nonuniformly hyperbolic}, if $ L(\alpha,A)>0 $, but $ (\alpha,A) $ is not uniformly hyperbolic,
    \item \textit{Critical}, otherwise.
\end{enumerate}
Applying this classification to Schr\"odinger operators $ H_{v,\alpha,\theta} $, we can also partition the energy axis into four regimes according to the classification of the corresponding Schr\"odinger cocycle: energies in the complement of the spectrum can be characterized as uniformly hyperbolic, and energies in the spectrum can be classified as subcritical, supercritical, and critical, respectively. The main result in \cite{avila2015global} is the following: for (measure-theoretically) typical $ v\in C^\omega(\T,\R) $, $ H_{v,\alpha,\theta} $ is acritical, that is, no energy $ E\in\Sigma $ is critical, where $ \Sigma $ denotes its spectrum. An interesting consequence is that acritical operators have the following spectrum bifurcation: There are $ k\geq 1 $ and points $ a_1<b_1<\ldots<a_k<b_k $ in the reslovent set such that $ \Sigma\subset\bigcup_{i=1}^k[a_i,b_i] $ and energies alternate between supercritical and subcritical along the sequence $ \{\Sigma \cap[a_i,b_i]\}_{i=1}^k $. Then, Avila's Almost Reducibility Conjecture (ARC),\footnote{A proof of the ARC was announced in \cite{avila2015global}, to appear in \cite{avila2010almost2,avila2010almost}. Indeed, the result was already stated as the \textit{Almost Reducibility Theorem} (ART) in \cite{AJM2017Spectral}.} which says that any subcritical cocycle is almost reducible, implies that the subcritical regime can only support absolutely continuous spectrum \cite{avila2010almost2}, and the supercritical regime supports pure point spectrum typically \cite{bourgain2000nonperturbative}, which allows us to conclude the typical absence of singular continuous spectrum. 

Our aim is to show that under an exponential decay assumption, this behavior is preserved under the perturbation by $g$, where $g$ is said to \emph{decay exponentially} if there is $\kappa > 0$ such that $e^{\kappa |n|} g(n) \in \ell^\infty(\Z)$.

In order to facilitate the formulation of the result we obtain, we consider a $ v\in C^\omega(\T,\R) $ for which the spectral description above holds and split $ \{1,\ldots,k\}= I_1 \sqcup I_2 $, where $ I_1 $ denotes the largest set such that for any $ i\in I $, the energies in $ \Sigma\cap[a_i,b_i] $ are subcritical. Let $ \Sigma^{sub}:=\bigcup_{i\in I_1}(\Sigma\cap[a_i,b_i]) $, and $ \Sigma^{sup}:=\bigcup_{i\in I_2}(\Sigma\cap[a_i,b_i]) $. 

We can now state our stability result in the one-frequency case.

\begin{theorem}\label{mainmain}
    For {\rm (}measure-theoretically{\rm )} typical $ v\in C^{\omega}(\T,\R) $ and exponentially decaying $ g $, we have the following: For every $ \theta\in\T $ and Lebesgue almost every $ \alpha \in \T$, the following holds for $ \widetilde{H}_{v,\alpha,\theta} $:
    \begin{enumerate}
        \item There is no singular continuous spectrum.
        \item The restriction of the canonical spectral measure of $ \widetilde{H}_{v,\alpha,\theta} $ to $ \Sigma^{sub} $ is purely absolutely continuous. 
        \item There are at most finitely  many eigenvalues in the restriction of each gap of $H_{v,\alpha,\theta} $ to $ \bigcup_{i\in I_1}[a_i,b_i] $.
        \item The restriction of the canonical spectral measure of $ \widetilde{H}_{v,\alpha,\theta} $ to  $ \bigcup_{i\in I_2}[a_i,b_i] $ is pure point and every eigenfunction decays exponentially.
    \end{enumerate}
\end{theorem}

\begin{remark}
(a) It is possible that $ \Sigma^{sup} $ does not contain all eigenvalues of $ \widetilde{H}_{v,\alpha,\theta} $ inside $ \bigcup_{i\in I_2}[a_i,b_i] $ because a decaying perturbation may {\rm (}and usually does{\rm )} produce eigenvalues in spectral gaps of the unperturbed operator. However, the accumulation points of the discrete eigenvalues of $ \widetilde{H}_{v,\alpha,\theta} $ in $ \bigcup_{i\in I_2}[a_i,b_i] $ must be in $\Sigma^{sup} $.
\\[1mm]
(b) On the other hand, since the perturbation preserves the essential spectrum and the discrete spectrum consists of discrete simple eigenvalues for which it is known that the associated eigenfunction decays exponentially, an equivalent formulation of item (3) is the following: The restriction of the canonical spectral measure of $ \widetilde{H}_{v,\alpha,\theta} $ to  $ \Sigma^{sup} $ is pure point and every eigenfunction decays exponentially.
\end{remark}

\begin{remark}
Let us comment on the difficulty and novelty aspects of the proofs, especially the proof of the persistence of purely absolutely continuous spectrum. It is not difficulty to obtain the result if the perturbation decays fast enough. However, it is a delicate issue to push the proof to obtain the result under the optimal decay assumption. Our method is based on ideas from dynamical systems, especially those developed by Avila in \cite{avila2010almost}. In this approach, one needs to prove that the set of spectral parameters with unbounded solutions has zero weight with respect to the spectral measure, which in turn relies on a measure estimate for the set of spectral parameters for which the corresponding cocycles have a given growth rate, and one seeks to establish suitable almost reducibility results (where one can conjugate into an arbitrarily small neighborhood of a constant) to control the growth of the cocycles. The  usual method \cite{avila2010almost, li2021absolutely, wang2021exact} is a Gronwall-type estimate  (first developed in Avila-Krikorian \cite [Lemma 3.1]{AK06reducibility} and also Avila-Fayad-Krikorian \cite[Claim 4.6]{krikorian2011kam}). However, using this traditional method, one can only prove the result for $g \in \ell^{2,1}$ (even this is not direct, one still needs to strengthen the estimates). Our method is different and explained in Lemma~\ref{new-gron} (see also Remark~\ref{rem:5.1}). Also our method to deal with the gap edges under the optimal decay condition is new (see Lemma~\ref{lem5.6}). The traditional way is to use the $ M $-function, this argument is also important to prove the result concerning finitely many eigenvalues in each of the gaps. 
To  prove the latter result, the key issues here are Proposition~\ref{positivesol} (positive ground state) and Proposition~\ref{conDep} (continuity of the Weyl solution). These two results were both proved quite recently, they depend on a recently developed new almost reducibility scheme \cite{leguil2017asymptotics, li2021absolutely}.
\end{remark}

\bigskip

The remainder of the paper is organized as follows. Section~\ref{sec.2} collects some important and well-known notions and results, which will be used in subsequent sections. Section~\ref{sec.3} establishes some growth and deviation estimates for perturbed transfer matrices. Theorem~\ref{mainresult} is proved in Section~\ref{sec.4}. Section~\ref{sec6} contains the proofs of part $(1)$ of Theorem \ref{acthm} and part $(2)$ of Theorem~\ref{mainmain}. Part $(2)$ of Theorem \ref{acthm} and part $(3)$ of Theorem~\ref{mainmain} are proved in Section~\ref{sec.6}. 
Finally, parts $(1)$ and $(4)$ of Theorem~\ref{mainmain} are proved in Section~\ref{sec.7}.

\section{Preliminaries}\label{sec.2}

\subsection{Quasi-Periodic Cocycles, Lyapunov Exponent, and Fibered Rotation Number}

Let $ \alpha\in\R^d $ with $ (1,\alpha) $ rationally independent.  A quasi-periodic cocycle $(\alpha, A)\in \R^d\times C^\omega(\T^d, \mathrm{SL}(2,\R))$  is a linear skew-product:
\begin{eqnarray*}\label{cocycle}
(\alpha,A):&\T^{d} \times \R^2 \to \T^{d} \times \R^2\\
\nonumber &(\theta,v) \mapsto (\theta+\alpha,A(\theta) \cdot v).
\end{eqnarray*}
Recall that a sequence $ (u_n)_{n\in\Z} $ is a formal solution of the eigenvalue equation $$ H_{v,\alpha,\theta}u=Eu $$ if and only if it satisfies $\left(\begin{matrix}
    u_{n+1}\\u_n
\end{matrix}\right)=S_E^v(\alpha,\theta,n)\cdot\left(\begin{matrix}
    u_n\\u_{n-1}
\end{matrix}\right) ,$ where
$$
    S_E^v(\alpha,\theta,n)=\begin{pmatrix}
        E-v(\theta+n\alpha)&-1\\1&0
    \end{pmatrix}\in\mathrm{SL}(2,\R).
$$
The map $ (\alpha,S_E^v) $ is called a \textit{Schr\"odinger cocycle}.

Denote the transfer matrix by 
\begin{align*}
    M_k(\alpha,\theta,E,n)&=S_E^v(\alpha,\theta ,n+k-1)\cdots S_E^v(\alpha,\theta, n),\end{align*}      
and $ M_{-k}(\alpha,\theta,E,n)=M_k(\alpha,\theta,E,n-k)^{-1} $. Define
\begin{equation*} 
    L_k(\alpha,E)=\frac{1}{k}\int_0^1 \log \lVert M_k(\alpha,\theta,E,0)\rVert \, d\theta.
\end{equation*} 
It is well-known that the Lyapunov exponent
\begin{equation*}
    L(\alpha,E)=\lim_{k\to\infty}L_k(\alpha,E)=\inf_k L_k(\alpha,E)\geq 0.
\end{equation*} 
is well-defined.

Assume $ A\in C^0(\T^d,\mathrm{SL}(2,\R)) $ is homotopic to the identity. Then there exist $ \psi:\T^d\times\T\to\R $ and $ u:\T^d\times\T\to\R_+ $ such that 
\[
    A(x)\cdot\begin{pmatrix}
        \cos 2\pi y\\\sin 2\pi y
    \end{pmatrix}=u(x,y)\begin{pmatrix}
        \cos 2\pi (y+\psi(x,y))\\\sin 2\pi (y+\psi(x,y))
    \end{pmatrix}.
\] 
The function $ \psi $ is called a \textit{lift} of $ A $. Let $ \mu $ be any probability measure on $ \T^d\times\R $ which is invariant by the continuous map $ T:(x,y)\mapsto(x+\alpha,y+\psi(x,y)) $, projecting over Lebesgue measure on the first coordinate. 
Then the \textit{fibered rotation number} of $ (\alpha,A) $ is defined by
\[
    \rho(\alpha,A)=\int\psi\ \dif\mu \!\!\!\! \mod \Z,
\] 
which does not depend on the choices of $ \psi $ and $ \mu $. It is easy to see that $ \rho(\alpha,S_E^v) $ admits a determination $ \rho_{v,\alpha}(E)\in[0,1/2] $. The 
{\it integrated density of states} (IDS) $N_{v,\alpha}\colon\R\to [0,1]$ of $H_{v,\alpha,\theta}$ is defined as
\[
N_{v, \alpha}(E):=\int_{\T^d} \mu_{v, \alpha,\theta}(-\infty,E] \, d\theta,
\]
where $\mu_{v,\alpha,\theta}$ is the spectral measure of $H_{v,\alpha,\theta}$ (and $\delta_0$).  
One can check that $ N_{v,\alpha}(E)=1-2\rho_{v,\alpha}(E) $; see \cite{avila2009ten, AS1983Almost, gaplabel}.

Recall that the  spectrum of $H_{v, \alpha, \theta}$, denoted by $\Sigma_{v, \alpha}$, is a compact subset of $\R$, independent of $\theta$ if $(1,\alpha)$ is rationally independent.
Any bounded connected component of $\R \backslash\Sigma_{v, \alpha}$ is called a \textit{spectral gap}.  


\subsection{Semi-Algebraic Sets}
    A set $ \mathcal{S}\subset \R^n $ is called a \textit{semi-algebraic set} if it is a finite union of sets defined by a finite number of polynomial equalities and inequalities. More precisely, let $ \{P_1,\cdots,P_s\}\subset\R[x_1,\cdots,x_n] $ be a family of real polynomials whose degrees are bounded by $ d $. A (closed) semi-algebraic set $ \mathcal{S} $ is given by an expression 
    \begin{equation}\label{semiS}
        \mathcal{S}=\bigcup_j\bigcap_{\ell\in \mathcal{L}_j}\{x\in\R^n:P_{\ell}(x)s_{j\ell}0\},
    \end{equation} 
    where $ \mathcal{L}_j\subset\{1,\cdots,s\} $ and $ s_{j\ell}\in\{\leq,\geq,=\} $ are arbitrary. Then we say that $ \mathcal{S} $ has degree at most $ sd $, and its degree is the infimum of $ sd $ over all representations as in (\ref{semiS}).

We introduce some useful facts about semi-algebraic sets; see, for example, \cite{bourgain2004green}.

\begin{lemma}[\cite{bourgain2004green} Tarski-Seidenberg principle]\label{TSprinciple}
If $ \mathcal{S}\subset\R^n $ is semi-algebraic of degree B, then any projection of $ \mathcal{S} $ is also semi-algebraic and of degree at most $ B^C $, where $ C $ only depends on $ n $.
\end{lemma}

\begin{lemma}[\cite{bourgain2004green}]\label{cor9.7inB}
Suppose that $ S\subset[0,1]^n $ is semi-algebraic of degree $ B $ and $ \meas_n S<\eta $. Let $ \alpha\in\T^n $ satisfy a DC and $ N $ be a large integer,
    $$
        \log B\ll \log N<\log \frac{1}{\eta}.
    $$ 
    Then for any $ \theta_0\in\T^d $
    \begin{equation}\label{9.8}
        \# \{k=1,\cdots,N|\theta_0+k\alpha\in \mathcal{S}( \mathrm{mod} 1)\}<N^{1-\delta}
    \end{equation} 
    for some $ \delta=\delta(\alpha) $.
\end{lemma}
\begin{lemma}[\cite{bourgain2004green}]\label{lem9.9inB}
    Let $ \mathcal{S}\subset[0,1]^{2n} $ be a semi-algebraic set of degree $ B $ and $ \meas_n \mathcal{S}<\eta $, $ \log B\ll\log\frac{1}{\eta} $. We denote $ (\omega,x)\in [0,1]^n\times[0,1]^n  $ the product variable. Fix $ \varepsilon>\eta^{\frac{1}{2n}} $. Then there is a decomposition 
    $$
        \mathcal{S}=\mathcal{S}_1\cup\mathcal{S}_2
    $$ 
    with $ \mathcal{S}_1 $ satisfying
    $$
        \meas_n (\mathrm{Proj}_{\omega}S_1)<B^{C(n)} \varepsilon
    $$ 
    and $ \mathcal{S}_2 $ satisfying the transversality property
    $$
        \meas_n(\mathcal{S}_2\cap L)\leq B^{C(n)}\varepsilon^{-1}\eta^{\frac{1}{2n}}
    $$ 
    for any $ n-$dimensional hyperplane s.t. 
    $$ 
        \max\limits_{0\leq j\leq n-1}|\mathrm{Proj}_L(e_j)|<\frac{1}{100}\varepsilon, 
    $$ 
    where we denote by $ e_0,\cdots,e_{n-1} $ the $ \omega$-coordinate vectors.
\end{lemma}
\subsection{Renormalized Oscillation Theory}
Let $ H $ be the Jacobi operator:
    \[
        (H u)(n)=a(n)u(n+1)+a(n-1)u(n-1)-b(n)u(n),
    \]  
    where $ a(n)<0 $, $ b(n)\in\R $, $ n\in\Z $. Let $ u $ be a solution of $ Hu=\lambda u $. A point $ n\in\Z $ is called a \textit{node} of solution $ u $ if either 
    \[
        u(n)=0 \text{ or }a(n)u(n)u(n+1)>0.
    \] 
    Denote by $ \sharp (u) $ the total number of nodes of $ u $ and by $ \sharp_{(m,n)}(u) $ the number of nodes of $ u $ between $ m $ and $ n $. More precisely, we shall say that a node $ n_0 $ of $ u $ lies between $ m $ and $ n $ if either $ m<n_0<n $ or if $ n_0=m $ but $ u(m)\neq 0 $. Let $ u_{1,2} $ be two solutions of $ H u_{1,2}=\lambda_{1,2}u_{1,2} $, respectively. The Wronskian of $ u_1 $, $ u_2 $ is denoted by
    \[
        W(u_1,u_2)(n)=a(n)\left(u_1(n)u_2(n+1)-u_1(n+1)u_2(n)\right).
    \]A point $ n\in\Z $ is called a \text{node} of $ W(u_1,u_2) $ if either
    \[
        W(u_1,u_2)(n)=0 \text{ or }W(u_1,u_2)(n)W(u_1,u_2)(n+1)<0.
    \] 
    Denote by $ \sharp W(u_1,u_2) $ the total number of nodes of $ \sharp W(u_1,u_2) $ and by $ \sharp_{(m,n)}W(u_1,u_2) $ the number of nodes of $ W(u_1,u_2) $ between $ m $ and $ n $. More precisely, we shall say that a node $ n_0 $ of $ W(u_1,u_2) $ lies between $ m $ and $ n $ if either $ m<n_0<n $ or if $ n_0=m $ but $ W(u_1,u_2)(m)\neq 0 $. 
    Let $ \dim \mathrm{Ran}P_{(\lambda_1,\lambda_2)} (H) $ denote the dimension of the range of the spectral projection $ P_{(\lambda_1,\lambda_2)}(H) $. Then $ \dim \mathrm{Ran}P_{(\lambda_1,\lambda_2)} (H)<\infty $ means there are at most finitely many eigenvalues in $ (\lambda_1,\lambda_2) $.  Then the central results of renormalized oscillation theory are the following.

\begin{theorem}[\cite{Tes1996Oscillation,Tes1999Jacobi}]\label{Teschl1}
    We have that
    \[
           \dim \mathrm{Ran}P_{(-\infty,\lambda)}(H)<\infty
    \] 
    if and only if one solution of $ H u=\lambda u $ satisfies: $ u(n) $ has a finite number of nodes; and we have
    \[
           \dim \mathrm{Ran}P_{(\lambda,\infty)}(H)<\infty
    \] 
    if and only if one solution of $ H u=\lambda u $ satisfies: $ (-1)^{n}u(n) $ has a finite number of nodes.
\end{theorem}

\begin{theorem}[\cite{Tes1996Oscillation,Tes1999Jacobi}]\label{Teschl2}
    Let $ \lambda_1<\lambda_2 $ and suppose $ u_{1,2} $ satisfy $ Hu_{1,2}=\lambda_{1,2}u_{1,2} $, then
    \[
         \begin{aligned}
             \sharp W(u_1,u_2)<\infty &\Leftrightarrow \dim \mathrm{Ran}P_{(\lambda_1,\lambda_2)}(H)<\infty.\\
         \end{aligned}
    \] 
\end{theorem}
\begin{theorem}[\cite{Tes1996Oscillation,Tes1999Jacobi}]\label{Teschl3}
 Let $ \lambda_1\leq \lambda_2 $ and suppose $ u_{1,2} $ satisfy $ Hu_{1,2}=\lambda_{1,2}u_{1,2} $, then 
 \[
     |\sharp_{(n,m)}W(u_1,u_2)-(\sharp_{(n,m)}(u_2)-\sharp_{(n,m)}(u_1))|\leq 2.
 \] 
\end{theorem}
For any $ \lambda\notin \sigma(H) $, there exist two linearly independent solutions $ u_{\pm}(\lambda) $ that are square-summable at $ \pm\infty $, respectively. If $ (\lambda_1,\lambda_2)\notin \sigma(H) $ with $ \lambda_1\text{ or }\lambda_2\in\sigma(H) $, then we can define $ u_{\pm}(\lambda_{1,2}) $ by a limit process. These solutions are called \textit{Weyl solutions} at $ \pm\infty $, respectively.
Then we have the following precise formula:
\begin{theorem}[\cite{Tes1999Jacobi,Tes1996Oscillation}]\label{Teschl4}
 Let $ \lambda_1<\lambda_2 $ and suppose $ [\lambda_1,\lambda_2]\cap\sigma_\mathrm{ess}(H)\subset\{\lambda_1,\lambda_2\} $. Let $ u_{\pm}(\lambda_1) $, $ u_{\pm}(\lambda_2) $ be the Weyl solutions of $ H $ at $ \lambda_1 $, $ \lambda_2 $ respectively, then
     \[
            \sharp W(u_{\pm}(\lambda_1),u_{\pm}(\lambda_2))= \dim \mathrm{Ran}P_{(\lambda_1,\lambda_2)}(H).
     \] 
\end{theorem}

\section{Upper Bound and Deviation Estimates}\label{sec.3}



For the perturbed operator, we consider its corresponding eigenvalue equation,
$$ 
\widetilde{H}_{v,\alpha,\theta}u=Eu.
$$
We have $ \left(\begin{matrix}
    u_{n+1}\\u_n
\end{matrix}\right)=\widetilde{S}_E^v(\alpha,\theta,n)\cdot\left(\begin{matrix}
    u_n\\u_{n-1}
\end{matrix}\right)$,  
where 
\begin{equation} 
    \widetilde{S}_E^v(\alpha,\theta,n)=\begin{pmatrix}
        E-v(\theta+n\alpha)-g(n)&-1\\1&0
    \end{pmatrix}\in\mathrm{SL}(2,\R).
\end{equation} 

For any $ v,E,\alpha $ fixed, denoting $ A(\theta,n)=S_E^{v}(\alpha,\theta,n) $ and $ B(\theta,n)=\widetilde{S}^{v}_E(\alpha,\theta,n) $, we consider for any $ k\geq 1 $ the (unperturbed and perturbed) transfer matrices given by 
\begin{align}
    M_k(\alpha,\theta,E,n)&=A(\theta,n+k-1)\cdots A(\theta,n),\\
    \widetilde{M}_k(\alpha,\theta,E,n)&=B(\theta,n+k-1)\cdots B(\theta,n).
\end{align}      
We also denote $ M_{-k}(\alpha,\theta,E,n)=M_k(\alpha,\theta,E,n-k)^{-1} $, $ \widetilde{M}_{-k}(\alpha,\theta,E,n)=\widetilde{M}_k(\alpha,\theta,E,n-k)^{-1} $.
In the following subsections, we establish estimates for $  \widetilde{M}_k(\alpha,\theta,E,n)$, which will be the starting point of the proof of our main results.

\subsection{Telescoping Arguments}
By a telescoping argument, or the method of variation of constants, we first establish a relationship between $M_k(\alpha,\theta,E,n)$ and $  \widetilde{M}_k(\alpha,\theta,E,n)$.

\begin{lemma}\label{perturbedM_k}
    We have the following:
    \begin{align}
        \widetilde{M}_{k}&(\alpha,\theta,E,n) \notag\\
        &=M_k(\alpha,\theta,E,n)+\sum_{i=0}^{k-1}\widetilde{M}_{k-i-1}(\alpha,\theta,E,n+i+1)\begin{pmatrix}
            -g(n+i)&0\\
            0&0
        \end{pmatrix}M_{i}(\alpha,\theta,E,n)\tag{3.1a}\label{3.1a}\\
         &=M_k(\alpha,\theta,E,n)+\sum_{i=0}^{k-1}M_{k-i-1}(\alpha,\theta,E,n+i+1)\begin{pmatrix}
          - g(n+i)&0\\
            0&0
        \end{pmatrix}\widetilde{M}_{i}(\alpha,\theta,E,n)\tag{3.1b}\label{3.1b}\\
        \widetilde{M}_{k}^{-1}&(\alpha,\theta,E,n)\notag \\
        &=M_k^{-1}(\alpha,\theta,E,n)+\sum_{i=0}^{k-1}M^{-1}_{i}(\alpha,\theta,E,n)\begin{pmatrix}
            0&0\\
            0& - g(n+i)
        \end{pmatrix}\widetilde{M}^{-1}_{k-i-1}(\alpha,\theta,E,n+i+1)\tag{3.2a}\label{3.2a}\\
        &=M_k^{-1}(\alpha,\theta,E,n)+\sum_{i=0}^{k-1}\widetilde{M}^{-1}_{i}(\alpha,\theta,E,n)\begin{pmatrix}
            0&0\\
            0&- g(n+i)
        \end{pmatrix}M^{-1}_{k-i-1}(\alpha,\theta,E,n+i+1)\tag{3.2b}\label{3.2b}
    \end{align}

\end{lemma}
\begin{proof}
We only prove \eqref{3.1a}, the proofs of the other formulas are similar. 
   By the telescoping argument, we have
    \begin{align*}
      \widetilde{M}_{k} & (\alpha,\theta,E,n) = \widetilde{A}(\theta,n+k-1)\cdots\widetilde{A}(\theta,n)\\
        & = {\scriptsize\begin{pmatrix}
            E-v(\theta+(n+k-1)\alpha)-g(n+k-1)&-1\\
            1&0
        \end{pmatrix}\cdots\begin{pmatrix}
            E-v(\theta+n\alpha)-g(n)&-1\\
            1&0
        \end{pmatrix}}\\
        & =  A(\theta,n+k-1)\cdots A(\theta,n) + \\
        & \qquad\sum_{i=0}^{k-1}\widetilde{A}(\theta,n+k-1)\cdots\widetilde{A}(\theta,n+i+1)\cdot\\
        & \qquad\qquad \left(\widetilde{A}(\theta,n+i)-A(\theta,n+i)\right)A(\theta,n+i-1)\cdots A(\theta,n)\\
        & = M_k(\alpha,\theta,E,n)+\sum_{i=0}^{k-1}\widetilde{M}_{k-i-1}(\alpha,\theta,E,n+i+1)\begin{pmatrix}
            -g(n+i)&0\\
            0&0
        \end{pmatrix}M_{i}(\alpha,\theta,E,n),
    \end{align*}
and the result follows.
\end{proof}

\subsection{Upper Bound}
In this subsection, we establish a uniform upper bound for $  \widetilde{M}_k(\alpha,\theta,E,n)$.
\begin{lemma}\label{expnumerator}
For any $ s>0 $, any $ g $ satisfying 
$$ 
|g(n)|\leq \mathcal{C}e^{-s|n|}, 
$$
any $ E $ in a compact set $ I \subset \R$, and any $ \epsilon>0 $, there exists $ k_{v}'(\mathcal{C},s,\epsilon,\alpha) $ such that 
    $$ 
    \lVert\widetilde{M}_k(\alpha,\theta,E,n)\rVert<e^{(L(\alpha,E)+\epsilon)k} 
    $$
    for all $ k>k_{v}'(\mathcal{C},s,\epsilon,\alpha) $, all $ \theta $, all $ n $.
\end{lemma}

\begin{proof}
First, we need to show that \eqref{M_k} can be bounded uniformly in $ \theta $, $ n $, and $ E $ in a compact set $ I $. We have the following lemma, which is essentially contained in \cite{avila2017sharp}. We give a proof for completeness.
\begin{lemma}\label{Upperb}
    For any $ \epsilon>0 $, there exist $ k_{v}(\epsilon,\alpha)<\infty $,  such that for any $ |k|>k_{v}(\epsilon,\alpha) $, we have
    $$
        \frac{1}{|k|}\log\lVert M_k(\alpha,\theta,E,n)\rVert\leq L(\alpha,E)+\epsilon
    $$ 
    uniformly in $ \theta $ and $ E $ in a compact set $ I $.
\end{lemma}
\begin{proof}
    We only consider the case where $ n $ is positive, the case of negative $n$ is similar. By Furman's result \cite{furman1997multiplicative}, for every $ E\in I $, and given $ \epsilon>0 $, there exists $ K=K(\epsilon,\alpha,E)>0 $, such that for any $ k>K $, we have
    $$
        \sup_{\theta\in\T^d}\frac{1}{|k|}\log\lVert M_k(\alpha,\theta,E,n)\rVert\leq L(\alpha,E)+\frac{\epsilon}{3}.
    $$ 
By the continuity of the Lyapunov exponent (cf. \cite{bourgain2005positivity}), there exists $ \delta=\delta(\epsilon,\alpha,E) $ such that if $ |E-E'|<\delta $, then 
    \begin{equation} \label{xxx}
        \sup_{\theta\in\T^d}\frac{1}{|k|}\log\lVert M_k(\alpha,\theta,E',n)\rVert\leq L(\alpha,E')+\frac{2\epsilon}{3},
    \end{equation} 
    holds for any $ K(\epsilon,\alpha,E)<k\leq 2K(\epsilon,\alpha,E)+1 $. By subadditivity, therefore (\ref{xxx}) holds for every $ k>K(\epsilon,\alpha,E) $.
    Since $ I $ is compact, by a compactness argument,  one can choose $ k_{v}=k_{v}(\epsilon,\alpha)<\infty $ such that for any $ |k|>k_{v}(\epsilon,\alpha) $, we have
    $$
        \sup_{\theta\in\T^d}\frac{1}{|k|}\log\lVert M_k(\alpha,\theta,E,n)\rVert\leq L(\alpha,E)+\epsilon,
    $$ 
    for all $E\in I$,
which completes the proof.
\end{proof}
By \eqref{3.1b}, we have 
    \begin{equation} \label{etismate1}
        \begin{aligned}
            \lVert \widetilde{M}_{k}(\alpha,\theta,E,n)\rVert\leq& \lVert M_k(\alpha,\theta,E,n)\rVert\\
            &+\sum_{i=0}^{k-1}|g(n+i)|\lVert M_{k-i-1}(\alpha,\theta,E,n+i+1)\rVert\lVert\widetilde{M}_{i}(\alpha,\theta,E,n)\rVert.
        \end{aligned}
    \end{equation} 
By Lemma \ref{Upperb}, for every $ E $ in a compact set $ I $, $ \epsilon>0 $, there exists $ k_{v}(\epsilon,\alpha) $ such that 
$$
    \lVert M_k(\alpha,\theta,E,n)\rVert\leq e^{(L(\alpha,E)+\epsilon)k+O(k_{v_0})}
$$ 
holds for all $ \theta $, all $ n $. Thus,
\begin{equation*}
    \begin{split}
        \lVert \widetilde{M}_{k}(\alpha,\theta,E,n)\rVert \leq &\quad e^{(L(\alpha,E)+\epsilon)k+O(k_{v_0})}\\
        &+\sum_{i=0}^{k-1}|g(n+i)|e^{(L(\alpha,E)+\epsilon)(k-i-1)+O(k_{v_0})}\lVert\widetilde{M}_{i}(\alpha,\theta,E,n)\rVert.
    \end{split}
\end{equation*}
Denoting
\begin{align*}
    c&=e^{-(L(\alpha,E)+\epsilon)+O(k_{v_0})}, \\
    y_k&= e^{-(L(\alpha,E)+\epsilon)k-O(k_{v_0})}\lVert \widetilde{M}_{k}(\alpha,\theta,E,n)\rVert,
\end{align*}
the bound \eqref{etismate1} can be rewritten as
\begin{equation} 
    y_k\leq 1+\sum_{i=0}^{k-1} c|g(n+i)|y_i.
\end{equation} 
By Gronwall's inequality, 
$$
    y_k\leq 1+\sum_{i=0}^{k-1}c|g(n+i)|\exp\left(\sum_{i<\ell<k}c|g(n+\ell)|\right)\leq C(\mathcal{C},s),
$$ 
that is 
$$
\lVert \widetilde{M}_{k}(\alpha,\theta,E,n)\rVert\leq C(\mathcal{C},s) e^{(L(\alpha,E)+\epsilon)k+O(k_{v_0})}.
$$ 
Thus there exists $ k_{v}'(\mathcal{C},s,\epsilon,\alpha) $ such that 
$$ 
\lVert\widetilde{M}_k(\alpha,\theta,E,n)\rVert<e^{(L(\alpha,E)+\epsilon)k} 
$$
for all $ k>k_{v}'(\mathcal{C} ,s,\epsilon,\alpha) $, all $ \theta $, all $ n $.
\end{proof}

\subsection{Deviation Estimates} As a consequence, we can obtain the following deviation estimates between $M_k(\alpha,\theta,E,n)$ and $  \widetilde{M}_k(\alpha,\theta,E,n)$.

\begin{lemma}\label{cor4}
    Assume that $ g $ is as in Lemma \ref{expnumerator}. For every $ E\in\R $, $ \epsilon>0 $, there exists $ k_{v}''(\mathcal{C},s,\epsilon,\alpha) $, such that for all $ \theta $ and any $ k>k_{v}'' $, 
    we have {\small
    \begin{align}
        \lVert\widetilde{M}_{k}(\alpha,\theta,E,n)-M_{k}(\alpha,\theta,E,n)\rVert &\leq e^{(L(E)+\epsilon)k-sn},& \text{ for } n&\geq 0,\tag{\ref{cor4}a}\label{4.3a}\\
        \lVert\widetilde{M}_{k}(\alpha,\theta,E,n)-M_{k}(\alpha,\theta,E,n)\rVert& \leq e^{(L(E)+\epsilon)k+s(n+k-1)},& \text{ for } n+k-1&\leq 0,\tag{\ref{cor4}b}\\
        \lVert\widetilde{M}_{- k}(\alpha,\theta,E,n)-M_{- k}(\alpha,\theta,E,n)\rVert&\leq e^{(L(E)+\epsilon)k-s(n-k)},&\text{ for } n-k&\geq 0,\tag{\ref{cor4}c}\\
        \lVert\widetilde{M}_{- k}(\alpha,\theta,E,n)-M_{- k}(\alpha,\theta,E,n)\rVert &\leq e^{(L(E)+\epsilon)k+s(n-1)},&\text{ for } n-1&\leq 0.\tag{\ref{cor4}d}
    \end{align}}
\end{lemma}
\begin{proof}
    We only prove \eqref{4.3a}, the proofs of the other estimates are similar. By \eqref{3.1b}, we have {\small
    \begin{equation} \label{etismate2}
        \begin{aligned}
            \lVert \widetilde{M}_{k}(\alpha,\theta,E,n)&-M_k(\alpha,\theta,E,n)\rVert\leq \sum_{i=0}^{k-1}|g(n+i)|\lVert M_{k-i-1}(\alpha,\theta,E,n+i+1)\rVert\lVert\widetilde{M}_{i}(\alpha,\theta,E,n)\rVert\\
            \leq &\sum_{i=0}^{k-1}|g(n+i)|\lVert M_{k-i-1}(\alpha,\theta,E,n+i+1)\rVert\lVert\widetilde{M}_{i}(\alpha,\theta,E,n)-M_i(\alpha,\theta,E,n)\rVert\\
            & +\sum_{i=0}^{k-1}|g(n+i)|\lVert M_{k-i-1}(\alpha,\theta,E,n+i+1)\rVert\lVert M_{i}(\alpha,\theta,E,n)\rVert.
        \end{aligned}
    \end{equation} }
    Similar to the proof of the previous lemma, for all $ \theta $, all $ n $, if we denote 
    \begin{align*}
        c&=e^{-(L(\alpha,E)+\epsilon)+O(k_{v_0})}, \\
        y_k&= e^{-(L(\alpha,E)+\epsilon)k-O(k_{v_0})}\lVert \widetilde{M}_{k}(\alpha,\theta,E,n)\rVert,
    \end{align*}
    then \eqref{etismate2} can be rewritten as 
\begin{equation} 
    y_k\leq c\sum_{i=0}^{k-1}|g(n+i)|+\sum_{i=0}^{k-1} c|g(n+i)|y_i.
\end{equation} 
By Gronwall's inequality, we have
$$
    y_k\leq c\sum_{i=0}^{k-1}|g(n+i)|+C(\mathcal{C},s)\sum_{i=0}^{k-1}c|g(n+i)|\exp\left(\sum_{i<\ell<k}c|g(n+\ell)|\right)\leq C'(\mathcal{C},s)e^{-sn},
$$ 
that is 
$$
\lVert \widetilde{M}_{k}(\alpha,\theta,E,n)-M_{k}(\alpha,\theta,E,n)\rVert\leq C'(\mathcal{C},s) e^{-sn}e^{(L(\alpha,E)+\epsilon)k+O(k_{v_0})}.
$$ 
Thus there exists $ k_{v}''(\mathcal{C},s,\epsilon,\alpha) $ such that 
$$ 
\lVert\widetilde{M}_k(\alpha,\theta,E,n)-M_k(\alpha,\theta,E,n)\rVert<e^{(L(\alpha,E)+\epsilon)k-sn} 
$$
for all $ k>k_{v}''(\mathcal{C} ,s,\epsilon,\alpha) $, all $ \theta $, all $ n $.
\end{proof}

\section{Anderson Localization}\label{sec.4}

In this section, we will prove the following Anderson localization result for the perturbed operator \eqref{pH}.
Denote 
$$ 
\mathcal{L}_\alpha(c_0):=\{E : L(\alpha,E)>c_0>0\}.
$$
By Shn\'ol's Theorem, it is enough to  show for almost all $ \alpha\in DC_d $, all generalized eigenvalues of $ \widetilde{H}_{v,\alpha,\theta_0} $ in $ \mathcal{L}_\alpha(c_0) $, have exponentially decaying eigenfunctions. 

The proof mainly follow the proof  in \cite{bourgain2004green}, we first introduce some useful notation. Denote
$$
    D_k(\alpha,\theta,n)={\begin{pmatrix}
        v(\theta+n\alpha) &1&&\\
        1 & v(\theta+(n+1)\alpha) &1&\\ 
        &&\ddots&&\\
        &&&1&v(\theta+(n+k-1)\alpha)
    \end{pmatrix}},
$$ 
and 
$$
    \widetilde{D}_k(\alpha,\theta,n)=D_k(\alpha,\theta,n)+
        \begin{pmatrix}
            g(n) &&&\\
            &g(n+1) &&\\ 
            &&\ddots&&\\
            &&&&g(n+k-1)
        \end{pmatrix}.
$$ 
Then we denote 
$$ 
    \widetilde{P}_k(\alpha,\theta,E,n)=\det[E-\widetilde{D}_k(\alpha,\theta,n)], \text{ for } k\geq 1,
$$ 
and set $ \widetilde{P}_0(\alpha,\theta,E,n)=1 $ and $ \widetilde{P}_{-1}(\alpha,\theta,E,n)=0 $.

The unperturbed operator, which corresponds to the case $ g\equiv 0 $, plays an important role in our argument. We denote the determinant that arises in this particular case by $ P_k(\alpha,\theta,E,n) $. 
Then it is easy to check that the transfer-matrix can be rewritten as
\begin{equation}\label{M_k}
    M_k(\alpha,\theta,E,n)=\begin{pmatrix}
        P_k(\alpha,\theta,E,n)&-P_{k-1}(\alpha,\theta,E,n+1)\\
        P_{k-1}(\alpha,\theta,E,n)&-P_{k-2}(\alpha,\theta,E,n+1)
    \end{pmatrix},
\end{equation}
and
\begin{equation}\label{MM_k}
    \widetilde{M}_k(\alpha,\theta,E,n)=\begin{pmatrix}
        \widetilde{P}_k(\alpha,\theta,E,n)&-\widetilde{P}_{k-1}(\alpha,\theta,E,n+1)\\
        \widetilde{P}_{k-1}(\alpha,\theta,E,n)&-\widetilde{P}_{k-2}(\alpha,\theta,E,n+1)
    \end{pmatrix}.
\end{equation}

We use the notation $ \widetilde{H}_{[N_1,N_2]}(\theta) $ for the operator $ \widetilde{H}_{v,\alpha,\theta} $ restricted to the interval $ [N_1,N_2] $ with zero boundary conditions at $ N_1-1 $ and $ N_2+1 $. Then we denote its Green's function as 
$$
G_{[N_1,N_2]}(E,\theta)(n_1,n_2)=(\widetilde{H}_{[N_1,N_2]}(\theta)-E)^{-1}(n_1,n_2),
$$
and Cramer's rule shows that for any $ N_1, N_2=N_1+N-1, N_1\leq n_1 \leq n_2\leq N_2 $, and $ E\notin\sigma(\widetilde{H}_{[N_1,N_2]}(\theta) ) $, we have
\begin{equation}\label{Greenfunction}
    |G_{[N_1,N_2]}(n_1,n_2)|=\frac{I_1}{I_2}
\end{equation} 
where 
\begin{equation}\label{I1I2}
    \begin{aligned}
        I_1&=|\widetilde{P}_{n_1-N_1}(\alpha,\theta,E,N_1)\widetilde{P}_{N_2-n_2}(\alpha,\theta,E,n_2+1)|, \\
        I_2&=|\widetilde{P}_N(\alpha,\theta,E,N_1)|.
    \end{aligned}
\end{equation} 
The following formula plays an elementary role in our proof:
\begin{lemma}
    Suppose $ n\in[N_1,N_2]\subset\Z $ and $ u $ is a solution of $ \widetilde{H}_{v,\alpha,\theta}u=Eu $. If $ E\notin\sigma(\widetilde{H}_{[N_1,N_2]}(\theta)) $, then
    $$
    u(n)=-G_{[N_1,N_2]}(E,\theta)(n,N_1)u(N_1-1)-G_{[N_1,N_2]}(E,\theta)(n,N_2)u(N_2+1).
    $$ 
\end{lemma}

\subsection{Green's Function Estimates}

All constants in what follows will depend on $ v $, but unless noted explicitly, we will leave this dependence implicit for simplicity. The main result of this section is the following:

\begin{proposition}\label{Prop7.1}
Suppose $I \subset \R$ is compact. Then, for any $ E\in \mathcal{L}_\alpha(c_0)\cap I $, there exists 
    $$
    \bar{N}_0=\bar{N}_0(c_0,\mathcal{C},s, \epsilon, \alpha)>0
    $$  
    such that if $ \alpha\in DC_d(\gamma,\tau) $, then for any $ N>\bar{N}_0 $ and any $ N'\geq N^2 $, the following holds:
    
    There is a set $ \Omega=\Omega(\alpha,E,N')\subset \T^d $, satisfying
    $$
        \meas \Omega<e^{-cN^{\sigma}}, \text{ for some }\sigma(\alpha)>0,\ c>0,
    $$ 
    such that for any $ \theta $ outside $ \Omega $, for any $ g $ satisfying $ |g(n)|\leq\mathcal{C}e^{-s|n|} $, one of the intervals
    $$
        \Lambda^R=[1,N]+N';[1,N-1]+N';[2,N]+N';[2,N-1]+N'
    $$ 
    and one of the intervals
    $$
        \Lambda^L=[-N,-1]-N';[-N+1,-1]-N';[-N,-2]-N';[-N+1,-2]-N'
    $$ 
    with both choices being independent of $ g $, will satisfy
    $$
        \max_{i\in\{L,R\}}|G_{\Lambda^{i}}(E,\theta)(n_1,n_2)|<e^{-L(\alpha,E)|n_1-n_2|+6\epsilon N }.
    $$ 
\end{proposition}
\begin{proof}
    We consider the case $ \Lambda^R $ first. For any $ \epsilon>0 $, by Lemma \ref{expnumerator}, for any interval $ \Lambda=[m+1,m+N] $ with length $ N\geq k_{v}'(\mathcal{C},s,\epsilon,\alpha) $, we have 
    \begin{equation}\label{G_Lambda}
        |G_{\Lambda}(E,\theta)(n_1,n_2)|<\frac{e^{(N-|n_1-n_2|)L(\alpha,E)+2\epsilon N+O(k_{v}')}}{|\widetilde{P}_N(\alpha,\theta,E,m+1)|}
    \end{equation} 
    for any $ g $ satisfying $ |g(n)|\leq\mathcal{C}e^{-s|n|} $.

    To get a lower bound for the denominator, we need the following Large Deviation Theorem (LDT) from Bourgain \cite{bourgain2004green}:
    \begin{theorem}[\cite{bourgain2004green}]
    Assume that $ \alpha\in\T^d $ satisfies the Diophantine condition $ DC_d(\gamma,\tau) $
    $$
        |k\alpha|_{\T^d}\geq \frac{\gamma}{|k|^{\tau}} \text{ for } k\in \Z^d\backslash\{0\}.
    $$ 
    Let $ v $ be real analytic on $ \T^d $. Fixing a small $ \epsilon>0 $ and taking $ N>N_0(\epsilon,\alpha) $, we have
    \begin{equation}\label{LDTmeas}
        \meas \left\{ \theta\in\T^d;\left|\frac{1}{N}\log\lVert M_N(\alpha,\theta,E,0)\rVert-L_N(\alpha,E)\right|\geq \epsilon \right\}\leq e^{-cN^\sigma},
    \end{equation} 
    for some constants $ \sigma=\sigma(\tau)>0 $, $ c=c(E)>0 $.
    \end{theorem}
    \begin{remark}
        $c=c(E)$ could be uniform for $ E $ in a bounded range, such as $ |E|\leq C\lVert v\rVert_0 $. Thus one can drop the dependence on $ E $.
    \end{remark}
    
From \eqref{LDTmeas} and translation invariance, it follows that if $ N>N_0(\epsilon,\alpha) $,
    \begin{equation}\label{LDTmeas'}
        \meas \left\{\theta\in\T^d;\left|\frac{1}{N}\log\lVert M_N(\alpha,\theta,E,N'+1)\rVert-L_N(\alpha,E)\right|\geq \epsilon \right\}\leq e^{-cN^\sigma},
    \end{equation}  
    Denote this set, which depends on $ \alpha,E,N' $, by $ \Omega^R $.
    By Corollary \ref{cor4}, we have if $ N>k_{v}''(\mathcal{C},s,\epsilon,\alpha) $, and large enough (depends on $ s $ but does not depend on $ E $ in a compact set), we have 
    \begin{equation}\label{Mn-Mn}
        \left\lvert\lVert\widetilde{M}_N(\alpha,\theta,E,N'+1)\rVert-\lVert M_N(\alpha,\theta,E,N'+1)\rVert\right\rvert\leq e^{(L(\alpha,E)+\epsilon-sN)N}<\tfrac{1}{2}
    \end{equation} 
    for any $ g $ satisfying $ |g(n)|\leq\mathcal{C}e^{-s|n|} $.
    Thus for any $ \theta\notin\Omega^R $, if $ N>\epsilon^{-1} $, then we have 
    $$
        \lVert\widetilde{M}_N(\alpha,\theta,E,N'+1)\rVert\geq e^{(L_N(\alpha,E)-\epsilon)N}-\tfrac{1}{2}\geq e^{(L(\alpha,E)-2\epsilon)N}.
    $$
    Clearly, the choice of $ \Lambda^R $ can be made independently of $ g $.
     
    Now we fix this $ \Lambda^R $. Let $ \bar{N}_0>\max\{N_0,k_{v}'^2,k_{v}'',\epsilon^{-1}\} $ and large enough, for any $ N>\bar{N}_0 $, then for any $ g $ satisfying $ |g(n)|\leq\mathcal{C}e^{-s|n|} $, the denominator in \eqref{G_Lambda} will be bounded from below by $ e^{(L(\alpha,E)-2\epsilon)N} $. Thus we have
    $$
    |G_{\Lambda^R}(E,\theta)(n_1,n_2)|<e^{-L(\alpha,E)|n_1-n_2|+4\epsilon N+O(k_{v}')}<e^{-L(\alpha,E)|n_1-n_2|+6\epsilon N}.
    $$ 
    The $ \Lambda^L $ case can be handled in a similar way by excluding a set $ \Omega^L $. Letting $ \Omega=\Omega^L\cup\Omega^R  $, we may conclude the proof. 
\end{proof}

\subsection{Proof of Theorem \ref{mainresult}}

\begin{proof}
      
    Let $ \epsilon=\tfrac{c_0}{1000} $. Fix $ \theta_0\in\T^d $, $ g $, and a generalized eigenvalue $E$ of $ \widetilde{H}_{v,\alpha,\theta_0} $ with a normalized generalized eigenfunction $ u=(u_n)_{n\in\Z}$, that is,\footnote{If the polynomially bounded solution happens to vanish at $0$, one can use $1$ as a reference site.}
$$
\widetilde{H}_{v,\alpha,\theta_0}u=Eu, \text{ with }|u_n|\leq C_E(1+|n|) \text{ and } |u_0|=1.
$$ 
First, we have the following lemma:
\begin{lemma}\label{lemdistEE'}
    For any $ N_0 $ sufficiently large, there is some $ j_0 $, $ |j_0|\leq 2N_1=2N_0^C $ where $ C $ a sufficiently large constant, such that 
    \begin{equation}\label{distEE'}
        0<\mathrm{dist} (E,\mathrm{spec}\widetilde{H}_{[-j_0+1,j_0-1]}(\theta_0))<e^{-\frac{c_0}{4}N_0}.
    \end{equation} 
\end{lemma}
\begin{proof}
    The proof of this lemma will be given in three steps.

    \noindent 
    \textbf{Step 1:}

    The properties of semialgebraic sets play a role in the Bourgain's approach \cite{bourgain2004green}. First, we need to show that the exceptional set could also be assumed semialgebraic. Following Bourgain's idea, one can shrink the exceptional set such that it is seen to be semialgebraic.
    \begin{claim}
        Fixing $ \alpha $ and $ N>\bar{N}_0 $ sufficiently large, the exceptional set $ \Omega $ in Proposition~\ref{Prop7.1} may be assumed to be semialgebraic of degree $ <N^3 $ and still independent of the choice of $ g $ satisfying $ |g(n)|\leq\mathcal{C}e^{-s|n|} $.
    \end{claim}
    \begin{proof}
    Fix $ \Lambda=\Lambda^{i}=[a,b] $ with $ i\in\{L,R\} $ and $ g $, we say that the pair $ (\theta,E) $ has the property ``$ g $-P '' if
\begin{equation} \label{gP}
    |G_\Lambda(E,\theta)(n_1,n_2)|<e^{-c_0|n_1-n_2|+6\epsilon N}, \text{ for all } n_1,n_2\in\Lambda.
\end{equation} 
Rewrite (\ref{gP}) as 
\begin{equation} \label{P}
    e^{2c_0|n_1-n_2|}[\det(\widetilde{H}_\Lambda-E)_{n_1,n_2}]^2\leq e^{2(6\epsilon N)}[\det(\widetilde{H}_\Lambda-E)]^2 , \text{ for all } n_1,n_2\in\Lambda,
\end{equation} 
where $ A_{n_1,n_2} $ denotes the $ (n_1,n_2) $-minor of the matrix $ A $. 

Observing that 
$$
\begin{aligned}
    |\det(\widetilde{H}_\Lambda-E)_{n_1,n_2}|&=|\widetilde{P}_{n_1-a}(\alpha,\theta,E,a)\widetilde{P}_{b-n_2}(\alpha,\theta,E,n_2+1)|\\
&\leq \lVert \widetilde{M}_{n_1-a}(\alpha,\theta,E,a)\rVert\cdot\lVert\widetilde{M}_{b-n_2}(\alpha,\theta,E,n_2+1)\rVert,
\end{aligned}
$$ 
one can replace \eqref{P} by the slightly stronger condition
\begin{equation} \label{P1}
    \sum_{n_1,n_2\in\Lambda}e^{2c_0|n_1-n_2|}\lVert \widetilde{M}_{n_1-a}(\alpha,\theta,E,a)\rVert^2\lVert \widetilde{M}_{b-n_2}(\alpha,\theta,E,n_2+1)\rVert^2\leq e^{2(6\epsilon N)}\left[\det(\widetilde{H}_\Lambda-E)\right]^2.
\end{equation}  
Since $ v_0 $ is analytic, we have $ v(\theta)=\sum_{k\in\Z^d}\hat{v}(k)e^{\mathrm{i}\langle k,\theta\rangle} $, $ |\hat{v}(k)|<e^{-\rho|k|} $, and we may substitute $ v $ by $ v_1=\sum_{|k|<C_1N}\hat{v}(k)e^{\mathrm{i}\langle k,\theta\rangle} $ where $ C_1=C_1(v) $ is a sufficiently large constant if we regard the deviation of $ v $ and $ v_1 $ as a perturbation. In fact, since $ \lVert v-v_1\rVert_{C^0}<e^{-C_1\rho N} $, choose $ C_1 $ sufficiently large such that 
$$
    C_1\rho N\gg C'(v,E).
$$   
If we use $ *^{\mathcal{T}}$ to denote quantities arising after this Fourier truncation, then by \eqref{3.1a}, a direct calculation shows that for any $ n\in\Lambda $, and $ k\leq N $,
$$
    \begin{aligned}
        |\widetilde{P}_k(\alpha,\theta,E,n)-P^{\mathcal{T}}_k(\alpha,\theta,E,n)|&\leq \sum_{j=0}^{k-1}\left(e^{-C_1 \rho N}+\mathcal{C}e^{-s(N^2+j)}\right)e^{C'(v,E)k}\\
        &\leq e^{C'(v,E)k}N (e^{-C_1\rho N}+\mathcal{C}e^{-sN^2}) \\
        &\leq e^{-\frac{1}{2}C_1\rho N}e^{C'(v,E)k}\\
        &\leq e^{-\frac{1}{3}C_1\rho N}
    \end{aligned}
$$ 
for any $ g $ satisfies $ |g(n)|\leq\mathcal{C}e^{-s|n|} $. 

Thus one can deduce the following estimates easily:
$$
    \begin{aligned}
        &|[\det(\widetilde{H}_\Lambda-E)]^2-[\det(H^{\mathcal{T}}_\Lambda-E)]^2|\leq e^{-\frac{1}{4}C_1\rho N},\\    
        &|\lVert \widetilde{M}_{n_1-a}(\alpha,\theta,E,a)\rVert^2\lVert\widetilde{M}_{b-n_2}(\alpha,\theta,E,n_2+1)\rVert^2\\
        &\qquad\qquad\qquad-\lVert M^{\mathcal{T}}_{n_1-a}(\alpha,\theta,E,a)\rVert^2\lVert M^{\mathcal{T}}_{b-n_2}(\alpha,\theta,E,n_2+1)\rVert^2|\\
        &\qquad\qquad\qquad\qquad\qquad\qquad\qquad\quad\;\leq e^{-\frac{1}{4}C_1\rho N}.
    \end{aligned}
$$ 
\begin{lemma}\label{abN}
    Suppose that $ \left\lvert|a|-|a'|\right\rvert\leq\epsilon_1 $ and $ \left\lvert|b|-|b'|\right\rvert\leq\epsilon_2 $. If
    $$
        |a'|\leq N|b'|-N\epsilon_2-\epsilon_1,
    $$ 
    then $ |a|\leq N|b| $.
\end{lemma}

\begin{proof}
This is easy to verify.
\end{proof}

Let $ N=e^{2(6\epsilon N)} $ and $ \epsilon_1=\epsilon_2=e^{-\frac{1}{4}C_1\rho N} $. By Lemma \ref{abN}, it is easy to see that for any $ g $ satisfying $ |g(n)|\leq\mathcal{C}e^{-s|n|} $, one can replace \eqref{P1} by the stronger condition
\begin{equation} \label{P''}
    \sum_{n_1,n_2\in\Lambda}e^{2c_0|n_1-n_2|}\lVert M^{\mathcal{T}}_{n_1-a}(\alpha,\theta,E,a)\rVert^2\lVert M^{\mathcal{T}}_{b-n_2}(\alpha,\theta,E,n_2+1)\rVert^2\leq e^{2(6\epsilon N)}\left[\det(H^{\mathcal{T}}_\Lambda-E)\right]^2-\tfrac{1}{2}.
\end{equation}  
Clearly, (\ref{P''}) is of the form 
\begin{equation} \label{P'''}
    P(\cos\alpha,\sin\alpha,\cos\theta,\sin\theta,E)\geq 0
\end{equation} 
where $ P $ is a polynomial of degree at most $ 2C_1 N^3 $.

Based on the same reasoning, one may again (assuming $ \theta $ bounded) truncate the power series for ``$ \cos $'' and ``$ \sin $'' and similarly replace (\ref{P''}) by the stronger condition (we use $ *^{\mathcal{T'}}$ to denote the corresponding quantities after the second truncation)
\begin{equation} \label{P'2}
    \sum_{n_1,n_2\in\Lambda}e^{2c_0|n_1-n_2|}\lVert M^{\mathcal{T'}}_{n_1-a}(\alpha,\theta,E,a)\rVert^2\lVert M^{\mathcal{T'}}_{b-n_2}(\alpha,\theta,E,n_2+1)\rVert^2\leq e^{2(6\epsilon N)}\left[\det(H^{\mathcal{T'}}_\Lambda-E)\right]^2-1,
\end{equation}  
 which is of the form
\begin{equation} \label{P'}
    P(\alpha,\theta,E)\geq 0
\end{equation} 
of degree at most $ N^5 $. Denote (\ref{P'2}) by the property ``$ 0^{\mathcal{T'}} $-P ''. 

Now we explain why the exceptional set $ \Omega $ in Proposition \ref{Prop7.1} may also be assumed semialgebraic of degree at most $N^5 $. Fix $ E $. Clearly, we have
$$
\{\theta : 0^{\mathcal{T'}}\text{-P}  \text{ holds for  one of the } \Lambda^i  \}\subset \bigcap_{|g(n)|\leq\mathcal{C}e^{-s|n|}}\{\theta : g\text{-P} \text{ holds for one of the } \Lambda^i \}.
$$
When $ \theta\notin\Omega $, which is defined in Proposition \ref{Prop7.1}, 
then there exists $ \Lambda $ such that 
$$
    \begin{aligned}
        \sum_{n_1,n_2\in\Lambda}&e^{2c_0|n_1-n_2|}\lVert M^{\mathcal{T'}}_{n_1-a}(\alpha,\theta,E,a)\rVert^2\lVert M^{\mathcal{T'}}_{b-n_2}(\alpha,\theta,E,n_2+1)\rVert^2\\
        &\leq\sum_{n_1,n_2\in\Lambda}e^{2c_0|n_1-n_2|}\lVert \widetilde{M}_{n_1-a}(\alpha,\theta,E,a)\rVert^2\lVert\widetilde{M}_{b-n_2}(\alpha,\theta,E,n_2+1)\rVert^2+e^{-c_1N}\\ 
        &\leq N^2e^{2(5\epsilon N)}e^{2(L(\alpha,E)-2\epsilon) N}+e^{-c_1N}-1\\
        &\leq N^2e^{2(5\epsilon N)}\left[\det(\widetilde{H}_\Lambda-E)\right]^2-\frac{1}{2}\\
        &\leq e^{2(6\epsilon N)}\left[\det(H^{\mathcal{T'}}_\Lambda-E)\right]^2-1,
    \end{aligned}
$$ 
that is, (\ref{P'2}) holds. Thus, 
$$
    \T^d\backslash\Omega\subset \{\theta : 0^{\mathcal{T'}}\text{-P}  \text{ holds for  one of the } \Lambda^i  \}.
$$ 
Denote $$ \Omega_1^{i}=\T^d\backslash \{\theta : 0^{\mathcal{T'}}\text{-P}  \text{ holds for } \Lambda^i  \}.
$$ 
It follows that
$$
\meas\bigcup_{i\in\{L,R\}}\bigcup_{|g(n)|\leq\mathcal{C}e^{-s|n|}} \{\theta : g\text{-P} \text{ does not hold for } \Lambda^i \} < \meas (\Omega_1^{L}\cup\Omega_1^{R})<\meas \Omega< e^{-cN^{\sigma}}.
$$ 
Thus we can shrink the exceptional set to be $ \Omega_1^{L}\cup\Omega_1^{R} $, which is  semialgebraic of degree $ <N^5 $.
\end{proof}

\noindent 
\textbf{Step 2:}

Fix $ N=N_0 $, $ N'=N_0^2 $ sufficiently large and let $ \Omega=\Omega(\alpha,E,N') $ be the set provided by  Proposition \ref{Prop7.1}. Consider the orbit $ \{\theta_0+j\alpha : 0\leq j\leq N_1\} $, $ N_1=N_0^{C}  $, where $ C $ is a sufficiently large constant. Applying Lemma~\ref{cor9.7inB} with $ \mathcal{S}=\Omega $, $ B=N_0^3 $, $ \eta=e^{-cN_0^{\sigma}} $ and $ N=N_1 $, it follows that except for at most $ N_1^{1-\delta} $ values of $ 0\leq j\leq N_1 $, $ \theta_0+j\alpha $ will not belong to $ \Omega $. For these $ j $, consider $ g=0 $, we have one of the intervals (which depends on $ \theta_0+j\alpha $)
\begin{equation} \label{Lambda}
    \Lambda^R=[1,N_0]+N_0^2;[1,N_0-1]+N_0^2;[2,N_0]+N_0^2;[2,N_0-1]+N_0^2
\end{equation} 
satisfies  (by the shift condition) 
$$
\begin{aligned}
    \sum_{n_1,n_2\in\Lambda^R+j}e^{2c_0|n_1-n_2|}\lVert M_{n_1-a}(\theta_0,a+j)\rVert^2&\lVert M_{b-n_2}(\theta_0,n_2+1+j)\rVert^2\\
    &\qquad\qquad\leq e^{2(6\epsilon N_0)}\left[\det(H_{\Lambda^R+j}-E)\right]^2.
\end{aligned}
$$ 
By \eqref{3.1a}, for any $ n $, we have 
$$
\begin{aligned}
    \widetilde{M}_{k}(\alpha,\theta,E,n)=&M_k(\alpha,\theta,E,n)+\\
    &\sum_{i=0}^{k-1}\widetilde{M}_{k-i-1}(\alpha,\theta,E,n+i+1)\begin{pmatrix}
        -g(n+i)&0\\
        0&0
    \end{pmatrix}M_{k-i}(\alpha,\theta,E,n+i)^{-1}\\
    &\ \cdot M_{k-i}(\alpha,\theta,E,n+i)M_{i}(\alpha,\theta,E,n).
\end{aligned}
$$ 
Thus we have for any $ k\leq N_0 $, for any $ g $ (satisfying $ |g(n)|\leq\mathcal{C}e^{-s|n|} $),
$$
    \begin{aligned}
       & \lVert \widetilde{M}_k(\alpha,\theta,E,N_0^2+1+j)-M_k(\alpha,\theta,E,N_0^2+1+j)\rVert \\
        &\leq \sum_{i=0}^{k-1}C'^{2k-2i-1}\mathcal{C} e^{-sN_0^2}\lVert M_k(\alpha,\theta,E,N_0^2+1+j)\rVert\\
        &\leq e^{-\frac{s}{2}N_0^2}\lVert M_k(\alpha,\theta,E,N_0^2+1+j)\rVert,
    \end{aligned}
$$ 
since the $ \Lambda^R+j $ we have chosen satisfies
$$
2|\det(H_{\Lambda^R+j}-E)|=\lVert M_k(\alpha,\theta,E,N_0^2+1+j)\rVert.
$$
Thus 
{\small
$$
\begin{aligned}
    &|[\det(\widetilde{H}_{\Lambda^R+j}-E)]^2-[\det(H_{\Lambda^R+j}-E)]^2|\leq e^{-\frac{s}{3}N_0^2}|\det(H_{\Lambda^R+j}-E)|^2,\\    
    &|\lVert \widetilde{M}_{n_1-a}(\alpha,\theta,E,a+j)\rVert^2\lVert\widetilde{M}_{b-n_2}(\alpha,\theta,E,n_2+1+j)\rVert^2\\
    &\qquad -\lVert M_{n_1-a}(\alpha,\theta,E,a+j)\rVert^2\lVert M_{b-n_2}(\alpha,\theta,E,n_2+1+j)\rVert^2|\\
    &\qquad\qquad\leq e^{-\frac{s}{3}N_0^2}\lVert M_{n_1-a}(\alpha,\theta,E,a+j)\rVert^2\lVert M_{b-n_2}(\alpha,\theta,E,n_2+1+j)\rVert^2.
\end{aligned}
$$ }
Thus for any $ g $, one of intervals $ \Lambda^R+j $, where $ \Lambda^R $ is defined as in \eqref{Lambda}, satisfies 
$$
\begin{aligned}
    \sum_{n_1,n_2\in\Lambda^R+j}&e^{2c_0|n_1-n_2|}\lVert \widetilde{M}_{n_1-a}(\alpha,\theta_0,E,a+j)\rVert^2\lVert \widetilde{M}_{b-n_2}(\alpha,\theta_0,E,n_2+1+j)\rVert^2\\
&\qquad\qquad\qquad\leq e^{2(7\epsilon N_0)}\left[\det(\widetilde{H}_{\Lambda^R+j}-E)\right]^2,
\end{aligned}
$$ 
which implies
$$
|G_{\Lambda^R+j}(E,\theta_0)(n_1,n_2)|<e^{-c_0|n_1-n_2|+7\epsilon N_0} \text{ for any $ g$ (satisfying $ |g(n)|\leq\mathcal{C}e^{-s|n|} $)}.
$$ 

Recall that if for some $ g $ fixed,
$$
    \widetilde{H}_{v,\alpha,\theta_0}u=Eu,\ \ |u_n|\leq C_E(1+|n|),
$$ 
then for any $ n\in\Lambda^R+j $, writing $ \Lambda^R+j=[n_1,n_2] $, we have
$$
    \begin{aligned}
        |u_n|&\leq |G_{\Lambda^R+j}(E,\theta_0)(n,n_1)| |u_{n_1-1}|+|G_{\Lambda^R+j}(E,\theta_0)(n,n_2)||u_{n_2+1}|\\
        &\leq C N_1 \max_{i\in\{1,2\}}e^{-c_0|n-n_i|+7\epsilon N_0}.
    \end{aligned}
$$ 
Taking in particular $ n=N_0^2+\tfrac{N_0}{2}+j $, then we have $ \min\limits_{i\in\{1,2\}}|n-n_i|\geq \tfrac{N_0}{3} $, thus 
$$
    |u_{N_0^2+\frac{N_0}{2}+j}|< e^{-\frac{c_0}{4}N_0}
$$ 
holds for all $ 0\leq j\leq N_1 $ except $ N_1^{1-\delta} $ many.
Similarly, we have 
$$
    |u_{-N_0^2-\frac{N_0}{2}-j}|<e^{-\frac{c_0}{4}N_0}
$$ 
holds for all $ 0\leq j\leq N_1 $ except $ N_1^{1-\delta} $ many. By the pigeonhole principle, one can find $ 0\leq j_0\leq 2N_1 $, such that
$$
    |u_{\pm j_0}|< e^{-\frac{c_0}{4}N_0}.
$$ 

\noindent
\textbf{Step 3:}

Let $ I=[-j_0+1,j_0-1] $, notice that \footnote{We denote the Hilbert-Schmidt norm of a matrix B as
\begin{equation*} 
    \lVert B\rVert_{\mathrm{HS}}=\left(\sum_{i,j}|B_{ij}|^2\right)^{\frac{1}{2}}.
\end{equation*}} 
$$
    \begin{aligned}
        1=|u_0|&\leq |G_{I}(E,\theta_0)(0,-j_0+1)| |u_{-j_0}|+|G_{I}(E,\theta_0)(0,j_0-1)||u_{j_0}|\\
        &\leq\lVert G_{[-j_0+1,j_0-1]}(E,\theta_0)\rVert_{\mathrm{HS}} e^{-\frac{c_0}{4}N_0}.
    \end{aligned}
$$ 
We may conclude that
$$
\lVert G_{[-j_0+1,j_0-1]}(E,\theta_0)\rVert_{\mathrm{HS}}> e^{\frac{c_0}{4}N_0},
$$ 
which is equivalent to 
\begin{equation*}
    0<\mathrm{dist} (E,\mathrm{spec} \widetilde{H}_{[-j_0+1,j_0-1]}(\theta_0))<e^{-\frac{c_0}{4}N_0}.
\end{equation*} 
\end{proof}

Next, let $ N_2=N_0^{C'} $, where $ C' $ is a sufficiently large constant, and denote 
$$
    \mathcal{E}=\mathcal{E}_{\alpha}=\cup_{j\leq 2N_1}\mathrm{spec}\widetilde{H}_{[-j+1,j-1]}(\theta_0).
$$ 
Suppose one can ensure that
\begin{equation}\label{keytoensure}
    \theta_0+n\alpha \ (\mathrm{mod }\  1)\notin \bigcup_{E'\in\mathcal{E}_\alpha}\Omega(\alpha,N_0^2,E') \text{ for all } N_2^{\frac{1}{2}}\leq |n|\leq 2N_2.
\end{equation} 
Then for any $ E'\in\mathcal{E}_\alpha $, and any $ N_2^{\frac{1}{2}}\leq n\leq 2N_2 $, there is an interval
$$
    \Lambda^R_{(n)}\in\{[1,N_0]+N_0^2,[1,N_0-1]+N_0^2,[2,N_0]+N_0^2,[2,N_0-1]+N_0^2\}
$$ 
such that
$$
    |G_{\Lambda^R_{(n)}+n}(E',\theta_0)(n_1,n_2)|<e^{-L(\alpha,E')|n_1-n_2|+7\epsilon N_0}<e^{-c_0|n_1-n_2|+7\epsilon N_0}.
$$ 
Recall that $ L(\alpha,E) $ is uniformly continuous in a compact set. We can therefore choose $ N_0 $ sufficiently large such that for any $ |E-E'|<e^{-\frac{c_0}{4}N_0} $, we have 
$$
    L(\alpha,E)\leq L(\alpha,E')+\epsilon.
$$ 
Then by Lemma~\ref{lemdistEE'} and a telescoping argument we see that
there is an interval
$$
\Lambda^R_{(n)}\in\{[1,N_0]+N_0^2,[1,N_0-1]+N_0^2,[2,N_0]+N_0^2,[2,N_0-1]+N_0^2\}
$$ 
such that
$$
    |G_{\Lambda^R_{(n)}+n}(E,\theta_0)(n_1,n_2)|<e^{-L(\alpha,E')|n_1-n_2|+8\epsilon N_0}<e^{-c_0|n_1-n_2|+8\epsilon N_0}.
$$ 
Define the interval
$$
    \widetilde{\Lambda}^R=\bigcup_{N_2^{\frac{1}{2}}\leq n\leq 2N_2}(\Lambda^R_{(n)}+n)\supset [N_2^{\frac{2}{3}},2N_2]
$$ 
and use the well-known ``paving property'' (Lemma 10.33 in \cite{bourgain2004green}) 
\noindent to find
$$
|G_{\widetilde{\Lambda}^R}(E,\theta_0)(n_1,n_2)|<e^{-(c_0-)|n_1-n_2|} \text{ if } |n_1-n_2|>\frac{1}{10}N_2,
$$ 
and therefore 
$$
    |u_j|<e^{-(c_0-)j} \text{ for }\frac{1}{2}N_2\leq j\leq N_2.
$$ 
The negative side is similar, thus we obtain the exponential decay.

The last step is to exclude an $ \alpha $-set $ \mathcal{R} $ with $ \meas\mathcal{R}=0 $ to ensure \eqref{keytoensure}, which is rather standard, thus Anderson localization holds. We refer to \cite{bourgain2004green} for details.
\end{proof}

\section{Absolutely Continuous Spectrum}\label{sec6}

As we explained in the introduction, $\ell^1$ perturbations are enough to preserve the ac spectrum $ \Sigma=\Sigma_{v,\alpha} $. In our setting, it just means that if $ (\alpha,S_E^v) $ is bounded, then $ (\alpha,S_E^{\tilde{v}}) $ is also bounded in the case of $\ell^1$ perturbations.  To prove that the perturbed operator has purely absolutely continuous spectrum in $\Sigma$, our method is to further explore the proof that the initial operator $H_{\lambda,\alpha,\theta} $ is purely ac, and reduce the desired statement to the following two statements:
\begin{itemize}

\item[(i):] $\widetilde{H}_{\lambda,\alpha,\theta} $ has no eigenvalues at the end of the gaps;

\item[(ii):] the energies $E\in  \Sigma$ for which $ (\alpha,S_E^v) $ is almost-reducible do not support any singular continuous spectrum of $\widetilde{H}_{\lambda,\alpha,\theta} $. 

\end{itemize}

Recall that a quasi-periodic cocycle $(\alpha,A)$ is \emph{analytically conjugated} to $(\alpha, \tilde{A})$ if there exists $B\in C^{\omega}(\T^d, {\rm PSL}(2,\R))$ such that 
$$
B(\cdot+\alpha)^{-1}A(\cdot)B(\cdot)=\tilde{A}(\cdot).
$$
A cocycle $(\alpha,A)$ is said to be \emph{almost reducible} if the closure of its analytic conjugacy class contains a constant. 

It will be crucial to estimate the growth of the cocycle $\lVert \widetilde{M}_k(\alpha,\theta,E,0)\rVert$. The method we employ is based on a KAM scheme. In this way one can obtain precise estimates for $\lVert M_k(\alpha,\theta,E,0)\rVert$, which we then parlay into a priori estimate for $\lVert \widetilde{M}_k(\alpha,\theta,E,0)\rVert$ in a controllable scale. The reader may consult Remark~\ref{rem:5.1} for more details.

\subsection{Growth of the Cocycle}

First we apply the KAM scheme from \cite{li2021absolutely,wang2021exact} to the Schr\"odinger cocycle $ (\alpha,S_E^v) $, where we  rewrite it as $ (\alpha,A_0 e^{f_0(\theta)}) $. Here, $ A_0(E)=\begin{pmatrix}\begin{smallmatrix}
    E &-1\\1&0
\end{smallmatrix}\end{pmatrix}
 $ and $ f_0=\begin{pmatrix}\begin{smallmatrix} 0& 0\\v & 0 \end{smallmatrix}\end{pmatrix} $. For any $\epsilon_0>0$ that is small enough (this will be defined later) and $r_0>0$, we inductively define
$$
    \epsilon_j=\epsilon_0^{2^j}, \qquad r_j=\frac{r}{2^j}, \qquad N_j=\frac{4^{j+1}\ln \epsilon_0^{-1}}{r}.
$$ 
Then we have the following:
 
\begin{proposition}[\cite{li2021absolutely,wang2021exact}]\label{Prop3.1}
Suppose $ \alpha\in DC_d(\gamma,\tau) $ and
    \[
        \lVert f_0\rVert_r \leq \epsilon_0\leq \frac{D_0}{ \sup_{E\in[-4,4]} \lVert A_0(E)\rVert^{C_0}} 
    \left(\frac{r}{2}\right)^{C_0\tau},
    \] 
where $ D_0=D_0(\gamma,\tau,d) $ and $ C_0 $ are numerical constants.  Then for any $ j\geq 1 $, there exists $ B_j\in C^\omega_{r_j}(2\T^d,\mathrm{SL}(2,\R)) $ such that 
    $$
        B_j(\theta+\alpha)^{-1}A_0(E) e^{f_0(\theta)}B_j(\theta)=A_je^{f_j(\theta)},
    $$ 
where $ A_j\in\mathrm{SL}(2,\R) $, $ \lVert A_j\rVert\leq 1+\epsilon_{j-1}^{\frac{1}{16}} $, $ \lVert f_j\rVert_{r_j}\leq \epsilon_{j} $ and $ B_j $ satisfies 
    $$
        \begin{aligned}
            \lVert B_j\rVert_0&\leq \epsilon_{j-1}^{-\frac{1}{192}},&
            |\deg B_j|&\leq 2 N_{j-1}.
        \end{aligned}
    $$ 
    \begin{enumerate}[label=$\mathit{(\arabic*)}$,align=left,widest=ii,labelsep=8pt,leftmargin=2em]
        \item For any $ 0<|n|\leq N_{j-1} $, denote 
        $$
            \Lambda_n(j)=\{E\in \Sigma:|2\rho(\alpha,A_{j-1}e^{f_{j-1}(\theta)})-\langle n,\alpha\rangle|_{\T}\leq \epsilon_{j-1}^{\frac{1}{15}}\}, \quad K_j:=\bigcup\limits_{0<|n|\leq N_{j-1}}\Lambda_n(j).
        $$ 
         Then for any $ E\in K_j $, there exists $ n_j\in\Z^d $ with $ |n_j|\leq 2N_{j-1} $ such that 
        \begin{equation}\label{xxxx}
            |2\rho(\alpha,A_0 e^{f_0(\theta)})-\langle n_j,\alpha\rangle|_{\T}<2\epsilon_{j-1}^{\frac{1}{15}}.
        \end{equation} 
        and we have
        \begin{equation} \label{xxxxx}
            \sup_{0<|s|<C\epsilon_{j-1}^{-\frac{1}{16}}}\lVert M_s(\alpha,\theta,E,0)\rVert_0 \leq C\lVert B_j\rVert_0^2\leq C \epsilon_{j-1}^{-\frac{1}{96}},
        \end{equation} 
        where $ C $ is a universal constant.
        \item Moreover, there always exist unitary matrices $ U_j\in\mathrm{SL}(2,\C) $ such that 
        \begin{equation}\label{ar}
            U_j^{-1}A_je^{f_j(\theta)}U_j=\begin{pmatrix}
                e^{2\pi i\rho_j}&c_j\\0&e^{-2\pi i\rho_j}
            \end{pmatrix}+F_j(\theta).
        \end{equation}
        Then for any $ E\in\Sigma $, we have $ \rho_j\in\R $, $ \lVert F_j\rVert_{r_j}\leq \epsilon_j $, $ \lVert B_jU_j\rVert^2_0 |c_j|\leq 8\lVert A_0\rVert $, and
        \begin{equation} \label{lineargrowth}
            \sup_{0<|n|<\epsilon_{j-1}^{-\frac{1}{16}}}\lVert M_n(\alpha,\theta,E,0)\rVert_0\leq 2\lVert B_jU_j\rVert^2_0(1+n |c_j|)\leq 2\epsilon_{j-1}^{-\frac{1}{96}}+16\lVert A_0\rVert |n|.
        \end{equation} 
\end{enumerate}
\end{proposition}

As a direct consequence, we have the following:

\begin{corollary}\label{uniform}
For any  $ E\in\Sigma $, we have 
\begin{equation} \label{uni-1}
        \lVert M_n(\alpha,\theta,E,0)\rVert_0 \leq C' |n|,
    \end{equation} 
    where $ C' $ is independent of $ E $.
\end{corollary}

\begin{proof}
This simple observation was first obtained by Eliasson \cite{eliasson}, we include the proof for completeness. 
 Indeed, since $ \epsilon_j=\epsilon_0^{2^j} $, for any $ n\in[\epsilon_{0}^{-\frac{1}{96}},+\infty) $, there exists $ j $ such that $ n\in [\epsilon_{j-1}^{-\frac{1}{96}},\epsilon_{j-1}^{-\frac{1}{16}}] $, then by (\ref{lineargrowth}), there exists $ C'>0 $ such that \eqref{uni-1} holds. 
\end{proof}

Once we have this, we can control the growth of the perturbed cocycle:
\begin{lemma}\label{new-gron}
If   $\sum_{n  } | n| |g(n)| <\infty$,  then for any  $ E\in\Sigma $, we have 
    \[
         \lVert \widetilde{M}_k(\alpha,\theta,E,0)\rVert\leq C'' k.
    \] 
Consequently, for any  $ E\in K_j $, we have
    \begin{equation}\label{lemC}
        \sup_{0<|s|<C\epsilon_{j-1}^{-\frac{1}{16}}}\lVert \widetilde{M}_s(\alpha,\theta,E,0)\rVert\leq C''' \epsilon_{j-1}^{-\frac{1}{96}},
    \end{equation} 
    where $ C''' $ is independent of $ E $.
\end{lemma}
\begin{proof}
We only consider the case where $ s $ is positive, the other case is similar. Recall that by (\ref{3.1b}), we have {\small
\[
    \begin{aligned}
        \widetilde{M}_{k}(\alpha,\theta,E,0)&=M_k(\alpha,\theta,E,0)+\sum_{i=0}^{k-1}M_{k-i-1}(\alpha,\theta,E,i+1)\begin{pmatrix}
            g(i)&0\\
            0&0
        \end{pmatrix}\widetilde{M}_{i}(\alpha,\theta,E,0).\\
    \end{aligned}
\] }
Thus we have
    \begin{equation} \label{5.6}
        \begin{aligned}
            \lVert \widetilde{M}_{k}(\alpha,\theta,E,0)\rVert \leq& \lVert M_{k}(\alpha,\theta,E,0)\rVert+\sum_{i=0}^{k-1}|g(i)|\lVert M_{k-i-1}(\alpha,\theta,E,i+1)\rVert\lVert\widetilde{M}_{i}(\alpha,\theta,E,0)\rVert.\\
        \end{aligned}
    \end{equation} 
 By Corollary \ref{uniform},   it follows that 
    \[
        \begin{aligned}
            k^{-1} \lVert \widetilde{M}_{k}(\alpha,\theta,E,0)\rVert \leq C'+C'\sum_{i=0}^{k-1}i|g(i)| i^{-1}\lVert\widetilde{M}_{i}(\alpha,\theta,E,0)\rVert.
        \end{aligned}
    \] 
By Gronwall's inequality this implies
    \[
         \lVert \widetilde{M}_k(\alpha,\theta,E,0)\rVert\leq \left( C'\exp\sum_{i=0}^{k-1} C' i |g(i)|\right)k\leq C'' k.
    \] 
Consequently, for any $ 0<s<C\epsilon_{j-1}^{-\frac{1}{16}} $, by \eqref{xxxxx} and  (\ref{5.6}), we have
\[
    \sup_{0<s<C\epsilon_{j-1}^{-\frac{1}{16}} }\lVert \widetilde{M}_s(\alpha,\theta,E,0)\rVert \leq C \epsilon_{j-1}^{-\frac{1}{96}}+C \epsilon_{j-1}^{-\frac{1}{96}}\sum_{i=0}^{s-1}C'' i|g(n+i)|\leq C''' \epsilon_{j-1}^{-\frac{1}{96}},
\] 
concluding the proof.
\end{proof}

\begin{remark}\label{rem:5.1}
Here the basic observation is Corollary \ref{uniform}. It is crucial to obtain the a priori bound on $ \lVert \widetilde{M}_k(\alpha,\theta,E,0)\rVert$. Otherwise, 
if we apply \eqref{ar} directly to estimate the growth of $ \lVert \widetilde{M}_k(\alpha,\theta,E,0)\rVert$, using traditional ways (see Avila-Krikorian \cite [Lemma 3.1]{AK06reducibility} and also Avila-Fayad-Krikorian \cite[Claim 4.6]{krikorian2011kam}), one can only obtain the result under the stronger assumption $g \in \ell^{2,1}$.
\end{remark}


\subsection{Proof of Theorem \ref{acthm} (1):}

  Let $ \mathcal{B} $ be the set of $ E\in\Sigma $ such that $ (\alpha,S^{v}_E) $ is bounded, and let $ \mathcal{B}' $ be the set of $ E\in\Sigma $ such that $ (\alpha,\widetilde{S}^{v}_E) $ is bounded, by (\ref{5.6}), it is easy to see that $ \mathcal{B}=\mathcal{B}' $. Recall the following well-known aspect of subordinacy theory:
  
\begin{theorem}[\cite{gilbert1987subordinacy,simon1996bounded}]\label{GP}

    Let $ \mathcal{B} $ be the set of $ E\in\sigma(\widetilde{H}_{v,\alpha,\theta}) $ such that $ (\alpha,\widetilde{S}^{v}_E) $ is bounded. Then the restriction of the spectral measure $ \widetilde{\mu}_{v,\alpha,\theta}|\mathcal{B} $ is purely absolutely continuous for all $ \theta\in\R $.
\end{theorem}

Thus by Theorem \ref{GP}, it is enough to prove that for every $ \xi\in\R $, $ \widetilde{\mu}=\widetilde{\mu}_{v,\alpha,\xi} $ is such that $ \widetilde{\mu}(\Sigma\backslash\mathcal{B}')=0 $.  Denote $ \mathcal{R} = \{ E \in \Sigma : (\alpha,S_E^{v}) \text{ is reducible}\} $. Observe that for every $ E \in \mathcal{R}\backslash\mathcal{B} $, $ (\alpha, S_E^v ) $ must be analytically reducible to a parabolic matrix. 

Let us recall the following well-known reducibility result of Eliasson:
  
\begin{theorem}\cite{eliasson}
Let $\delta > 0$, $\alpha \in DC_d(\gamma,\tau) $, and $A_{0}\in \mathrm{SL}(2,\mathbb{R})$. Then there is a constant $\epsilon = \epsilon(\gamma,\tau,\delta,\| A_{0}\|)$ such that if $A \in C^{\omega}_{\delta}(\mathbb{T}^{d},\mathrm{SL}(2,\mathbb{R}))$ is real analytic with
$$
\| A - A_{0} \|_{\delta} \leq \epsilon
$$
and the rotation number of the cocycle $(\alpha, A)$ satisfies
$$
\|2\rho(\alpha,A) - \langle n, \alpha \rangle\|_{\mathbb{R}/\mathbb{Z}} \geq \frac{\kappa}{|n|^{\tau}} \quad \forall \;  0 \neq n \in \mathbb{Z}^{d}
$$
or $2\rho(\alpha,A) = \langle n, \alpha \rangle \mod \Z$ for some $n \in \mathbb{Z}^{d}$, then $(\alpha,A)$ is analytically reducible.
\end{theorem}

It follows that $ \mathcal{R}\backslash\mathcal{B} $ is countable: indeed for any such $E$, the well known gap labeling theorem ensures that there exists a $k\in \Z^{d}$ such that $2\rho (\alpha,S^{v}_E) =\langle k,\alpha \rangle \mod \mathbb{Z}$. We further prove $E\in \mathcal{R}\backslash\mathcal{B} $ is not an eigenvalue of $ \widetilde{H}_{v,\alpha,\theta}$.
The following observation is important for us, a similar idea appears in Coddington-Levinson \cite[Chapter 3, Exercise 35]{CL1994Theory}.

        \begin{lemma}\label{lem5.6}
     Let $c\neq 0$ and suppose that 
            \begin{equation} \label{5.7}
                \vec{\varphi} (n+1)=\left(A+R(n)\right)\vec{\varphi}(n),
            \end{equation}  
            where
            \[
               A=\begin{pmatrix}
                    \pm 1&c\\
                    0&\pm 1
                \end{pmatrix}, \qquad \sum |R(n)||n|<\infty.
            \]
           Then any non-zero solution of (\ref{5.7}) does not tend to zero when $ n\to \infty $. Moreover, for any $ \epsilon>0 $, we can find a solution $ \vec{\phi}_{\pm}(n) $ such that 
           \[
            \limsup_{n\to\pm\infty}\|\vec{\phi}_{\pm}(n)-\begin{pmatrix}
                \pm 1\\0
            \end{pmatrix}\|\leq \epsilon.
           \]  
        \end{lemma}
        \begin{proof}
           Without loss of generality, we consider $ A=\begin{pmatrix}1&c\\0&1\end{pmatrix} $. Denote by $ \Phi(n) $ the  fundamental matrix of $ \vec{\varphi}(n+1)=A\vec{\varphi}(n) $ and decompose
            \[
                \begin{aligned}
                    \Phi(n-s)&=\begin{pmatrix}
                        1&c(n-s)\\
                        0&1
                    \end{pmatrix}=\begin{pmatrix}
                        1&-cs\\
                        0&0
                    \end{pmatrix}+\begin{pmatrix}
                        0&cn\\
                        0&1
                    \end{pmatrix}:=\Phi_1(s)+\Phi_2(n).
                \end{aligned}
            \] 
            It follows that
            \[
                \begin{aligned}
                    \|\Phi_1(s)\|&\leq K_1 s,\  s\geq 1,\\
                \|\Phi_2(n)\|&\leq K_2 n,\  n\geq 1.
                \end{aligned}
            \] 
            Let $ \vec{\psi}_0(n)=\Phi(n)\begin{pmatrix}
                0\\1
            \end{pmatrix} = \begin{pmatrix}
               cn\\1
            \end{pmatrix}  $, $ \ n\geq 1 $. One has
            \[
                \|\vec{\psi}_0(n)\|\leq K_0 n.
            \]

          Choose  $ n_0 $ which is large  enough such that
            \[
                (K_1+K_2)\sum_{s=n_0}^{\infty}|R(s)|s<\frac{1}{2},
            \] 
and define the sequence
            \[
                \vec{\psi}_{i+1}(n)=\begin{pmatrix}cn\\1\end{pmatrix} +\sum_{s=n_0+1}^n\Phi_1(s)R(s-1)\vec{\psi}_{i}(s-1)-\sum_{s=n+1}^\infty\Phi_2(n)R(s-1)\vec{\psi}_{i}(s-1).
            \] 
A direct computation shows that
            \[
                \begin{aligned}
                    \|\vec{\psi}_{i+1}(n)-\vec{\psi}_{i}(n)\|&\leq \frac{K_0 n}{2^{i+1}},
                \end{aligned}
            \]
            which means that there exists a limit function $ \vec{\psi}(n) $ on  $ n\geq n_0 $ that satisfies
            \[
                \|\vec{\psi}(n)\|\leq 2K_0 n,
            \] 
            \begin{equation}\label{iter}
                \vec{\psi}(n)=\begin{pmatrix}cn\\1\end{pmatrix}+\sum_{s=n_0+1}^n\Phi_1(s)R(s-1)\vec{\psi}(s-1)-\sum_{s=n+1}^\infty\Phi_2(n)R(s-1)\vec{\psi}(s-1).
            \end{equation}
            It is easy to verify that  $ \vec{\psi}(n)$ is a solution of \eqref{5.7}. Next we give an estimate for $ \vec{\psi}(n)$. 
           
      By \eqref{iter}, we have
            \[
                \begin{aligned}
                    n^{-1}\|\vec{\psi}(n)-\begin{pmatrix}cn\\1\end{pmatrix}\| & \leq 2K_0K_1\left(\frac{1}{\sqrt{n}}\sum_{s=n_0}^{\left\lfloor \sqrt{n}\right\rfloor +1}|R(s)|s+\sum_{s=\left\lfloor \sqrt{n}\right\rfloor+2}^{n-1}|R(s)|s\right)\\
                    & \quad + 2K_0K_2\sum_{s=n}^\infty |R(s)|s\\
                    & \leq 2K_0K_1 \frac{1}{\sqrt{n}}\sum_{s=n_0}^{\left\lfloor \sqrt{n}\right\rfloor +1}|R(s)|s+2K_0(K_1+K_2)\sum_{s=\left\lfloor \sqrt{n}\right\rfloor+2}^{\infty}|R(s)|s\to 0.
                \end{aligned}
            \]

       If we let $ \vec{\phi}_0(n)=\Phi(n)\begin{pmatrix}
                1\\0
            \end{pmatrix} =\begin{pmatrix}
                1\\0
            \end{pmatrix} $, $ n\geq 1 $, we have
            \[
                \|\vec{\phi}_0(n)\|\leq K_0 .
            \]
            Following the same reasoning as above, if we choose $ n_0 $ such that
            \[
                K_0(K_1+K_2)\sum_{s=n_0}^{\infty}|R(s)|(s+1)<\epsilon,
            \] 
            we see that there exists a limit function $ \vec{\phi}(n) $ which satisfies
            \[
                \|\vec{\phi}(n)\|\leq 2K_0 , \text{ for } n\geq n_0,
            \] 
            \[
                \vec{\phi}(n)=\begin{pmatrix}1\\0\end{pmatrix}+\sum_{s=n_0+1}^n\Phi_1(s)R(s-1)\vec{\phi}(s-1)-\sum_{s=n+1}^\infty\Phi_2(n)R(s-1)\vec{\phi}(s-1),
            \] 
         and which is also a solution of 
         \eqref{5.7}.  A direct computation shows that
            \[
                \|\vec{\phi}(n)-\begin{pmatrix}
                    1\\0
                \end{pmatrix}\|< \epsilon, \text{ for } n\geq n_0.
            \] 
    The proof of the statement on the left half line is similar.
    Finally, it is easy to check that $ \vec{\psi}(n) $ and $ \vec{\phi}(n) $ are linearly independent. This completes the proof.
        \end{proof}

Suppose $ (\alpha,S_E^v) $ is analytically reducible to a parabolic matrix, that is, there exists $ B\in C^\omega(2\T^d,\mathrm{SL}(2,\R)) $ such that 
        \[
            B(\theta+(n+1)\alpha)^{-1}S_E^v(\theta+n\alpha)B(\theta+n\alpha)=\begin{pmatrix}
                \pm 1&c\\
                0&\pm 1
            \end{pmatrix}.
        \] 
        Then the second order difference equation
        \[
            \left(\begin{matrix}
                \tilde{u}_{n+1}\\\tilde{u}_n
            \end{matrix}\right)=\widetilde{S}_E^v(\alpha,\theta,n)\cdot\left(\begin{matrix}
                \tilde{u}_n\\\tilde{u}_{n-1}
            \end{matrix}\right)
        \] 
        can be transformed to
        \[
            \vec{\varphi}(n+1)=\left(\begin{pmatrix}
                \pm 1&c\\
                0&\pm 1
            \end{pmatrix}+R(n)\right)\cdot \vec{\varphi}(n)
        \] 
The assumption  $\sum_{n  } | n| |g(n)| <\infty$ implies that $\sum |R(n)||n|<\infty$. It then follows from Lemma~\ref{lem5.6} that any non-zero solution $ \widetilde{H}_{v,\alpha,\xi}\widetilde{u}=E\widetilde{u} $ satisfies $$ 
\inf_{n\in\Z^{+}}|\widetilde{u}_{n+1}|^2+|\widetilde{u}_{n}|^2>0,
$$
that is, there are no eigenvalues in $\mathcal{R}\backslash\mathcal{B} $.

Thus it remains to prove that $ \widetilde{\mu}(\Sigma\backslash\mathcal{R})=0 $. To prove this, we recall the following result of Avila \cite{avila2008absolutely}, which is essentially  the Jitomirskaya-Last inequality \cite{jitomirskaya1999power}:
\begin{lemma}\label{lem3.3}
We have 
$$
\widetilde{\mu}(E-\epsilon,E+\epsilon)\leq C\sup_{0<|s|<C\epsilon^{-1}}\lVert \widetilde{M}_s(\alpha,\theta,E,0)\rVert_0^2, 
$$
where $ C>0 $ is a universal constant and $ \widetilde{\mu} $ is the canonical spectral measure of $ \widetilde{H}_{v,\alpha,\theta} $.
\end{lemma}

For $ E\in K_m $, let $ J_m(E) $ be an open $ 2^{\frac{2}{3}}\epsilon_{m-1}^{\frac{2}{45}} $-neighborhood of $ E $. By \eqref{lemC}, we have 
$$
\sup_{0<|s|<C\epsilon_{m-1}^{-\frac{2}{45}}}\lVert \widetilde{M}_s(\alpha,\theta,E,0)\rVert_0\leq C''\epsilon_{m-1}^{-\frac{1}{96}}, 
$$
and therefore, by Lemma \ref{lem3.3}, we have
$$
    \widetilde{\mu}(J_m(E))\leq C''\epsilon_{m-1}^{-\frac{1}{48}}|J_m(E)|,     
$$ 
where $ |\cdot| $ denotes Lebesgue measure. Take a finite subcover $ \overline{K_m}\subset \cup_{j=0}^r J_m(E_j) $. Refining this subcover if necessary, we may assume that every $ x \in \overline{K_m} $ is contained in at most $ 2 $ different $ J_m(E_j) $. 
    
    For $ E\in K_m $, (\ref{xxxx}) implies that $ 2\rho(K_m) $ can be covered by $ 2N_{m-1} $ intervals $ T_s $ of length $ 2\epsilon_{m-1}^{\frac{1}{15}} $.  Recall the following lower bound for the integrated density of states $ N(E) $ of the unperturbed operator $ H_{v,\alpha,\theta} $:
     
\begin{lemma}[\cite{avila2008absolutely}]\label{lem3.11}
    If $ E\in\Sigma $, then for $ 0<\epsilon<1 $, $ N(E+\epsilon)-N(E-\epsilon)\geq c\epsilon^{\frac{3}{2}} $.
\end{lemma}
  
    By Lemma \ref{lem3.11}, $ |N(J_m(E))|\geq c|J_m(E)|^{\frac{3}{2}} $, we have $ |T_s|\leq \frac{1}{c}|2\rho(J_m(E))| $ for any $ s $, $ E\in K_m $. We conclude that there are at most $ r\leq (2([\frac{1}{c}]+1)+4)N_{m-1} $ intervals $ T_s $ to cover $ K_m $, i.e. $ r\leq CN_{m-1} $. Then 
    $$
        \widetilde{\mu}(\overline{K_m})\leq \sum_{j=0}^{r}\widetilde{\mu}(J_m(E_j))\leq CN_{m-1}C''\epsilon_{m-1}^{\frac{1}{48}},
    $$ 
    which gives $ \sum_m \widetilde{\mu}(\overline{K_m})<\infty $, then by Borel-Cantelli lemma, $$ \widetilde{\mu}(\Sigma\backslash\mathcal{R})\leq \widetilde{\mu}(\limsup K_m)=0 ,$$
and the result follows.
\subsection{Proof of  Theorem \ref{mainmain} (2):}
In this subsection we prove (2) of Theorem \ref{mainmain}. Avila's \emph{almost reducibility conjecture} (ARC) will play a role. The full solution of the ARC was recently given by Avila in \cite{avila2010almost2, avila2010almost}.

\begin{theorem}[\cite{avila2010almost2, avila2010almost}]\label{ART}
    Given $ \alpha\in \R\backslash\Q $ and $ A\in C^\omega(\T,\mathrm{SL}(2,\R)) $, if $ (\alpha,A) $ is subcritical, then it is almost reducible.
\end{theorem}

Choose $ M>0 $ such that $ \Sigma\subset (-M,M) $ and let 
$$
    \mathcal{AR}=\{E\in (-M,M) : (\alpha,S_E^v) \text{ is  almost reducible}\}.
$$
As $\mathcal{AR}$ is open (Corollary 1.3 of \cite{avila2010almost}), we can write $ \mathcal{AR}=\bigcup_{j=1}^J (a_j,b_j) $, where $ J $ may be finite or countable. Take any $ (a_j,b_j) $ in $ \bigcup_{j=1}^J (a_j,b_j) $, and denote it by $ (a,b) $. Define
$$
    \mathcal{S}(\delta_0)=[a+\delta_0,b-\delta_0]
$$ 
for any sufficiently small $ \delta_0>0 $. Then we have the following:

\begin{lemma}[\cite{wang2021exact}]\label{globaltolocal}
    For any $ \epsilon_0>0 $, $ \alpha\in\R\backslash\Q $, there exist $ \bar{h}=\bar{h}(\alpha)>0 $ and $ \Gamma=\Gamma(\alpha,\epsilon_0)>0 $ such that for any $ E\in \mathcal{S}(\delta_0) $, there exists $ \Phi_E\in C^\omega(\T,\mathrm{PSL(2,\R)}) $ with $ \lVert\Phi_E\rVert_{\bar{h}}<\Gamma $ such that
    $$
        \Phi_E(\theta+\alpha)^{-1}S_E^v(\theta)\Phi_E(\theta)=R_{\Phi_E}e^{f_E(\theta)}
    $$ 
    with $ \lVert f_E\rVert_{\bar{h}}<\epsilon_0 $, $ |\deg \Phi_E |\leq C|\ln\Gamma| $ for some constant $ C=C(v,\alpha)>0 $.
\end{lemma}

\begin{proof}
    This is essentially contained in Lemma 4.2 of \cite{wang2021exact}. The crucial fact for this lemma is that we can choose $ \bar{h}(\alpha) $ to be independent of $ E $ and $ \epsilon_0 $, and choose $ \Gamma $ to be independent of $ E $. We refer to Propositions 5.1 and 5.2 of \cite{leguil2017asymptotics} for details.
\end{proof}

Once we have Lemma \ref{globaltolocal}, after a finite number of conjugation steps, which is uniform in $ E\in\mathcal{S}(\delta_0) $, we can reduce the cocycle to the perturbative regime. Then one can apply the KAM scheme to get precise control of the growth of the cocycle. Define 
$$
    r=\bar{h},\qquad \epsilon_0\leq D_0\left(\frac{r}{2}\right)^{C_0\tau},\qquad\epsilon_j=\epsilon_0^{2^j}, \qquad r_j=\frac{r}{2^j}, \qquad N_j=\frac{4^{j+1}\ln \epsilon_0^{-1}}{r},
$$ 
and replace $ \Sigma $ by $ \mathcal{S}(\delta_0) $. Then (\ref{xxxxx}) follows from $ \lVert\Phi_E\rVert_{\bar{h}}<\Gamma $, while (\ref{xxxx}) will follow from $ |\deg \Phi_E |\leq C|\ln\Gamma| $. 

Indeed, by Lemma \ref{globaltolocal} and Proposition \ref{Prop3.1}, there exist $ \Phi_E\in C^\omega_r(2\T,\mathrm{SL}(2,\R)) $ and $ B_j\in C^\omega_{r_j}(2\T,\mathrm{SL}(2,\R)) $ with $ |\deg B_j|\leq 2N_{j-1} $ such that 
$$
    \begin{aligned}
        B_{j}(\theta+\alpha)^{-1}\Phi_E(\theta+\alpha)^{-1}S_E^v\Phi_E(\theta)B_j(\theta)=B_{j}(\theta+\alpha)^{-1}A_0(E)e^{f_0(\theta)}B_j(\theta)=A_j e^{f_j(\theta)},
    \end{aligned}
$$ 
with $ A_j\in\mathrm{SL}(2,\R) $, $ \lVert A_j\rVert\leq 1+\epsilon_{j-1}^{\frac{1}{16}} $, $ \lVert f_j\rVert_{r_j}\leq \epsilon_{j} $ and $ \lVert B_j\rVert_0\leq \epsilon_{j-1}^{-\frac{1}{192}} $. Then for any $ E\in K_j $, there exists $ m\in\Z^d $ with $ 0<|m|\leq N_{j-1} $ such that 
\begin{equation*}
    |2\rho(\alpha,A_{j-1} e^{f_{j-1}(\theta)})-\langle m,\alpha\rangle|_{\T}<2\epsilon_{j-1}^{\frac{1}{15}}.
\end{equation*} 
Notice that
$$ 
    2\rho(\alpha,S_E^v)=2\rho(\alpha,A_{j-1} e^{f_{j-1}(\theta)})+\langle\deg B_{j-1}+\deg \Phi_E,\alpha\rangle,
$$
let $ n_j=\deg B_{j-1}+\deg\Phi_E+m $, we have $ |n_j|\leq 2N_{j-2}+N_{j-1}+C|\ln\Gamma|\leq 2N_{j-1} $, and 
\begin{equation} 
    |2\rho(\alpha,S_E^v)-\langle n_j,\alpha\rangle|_{\T}<2\epsilon_{j-1}^{\frac{1}{15}}.
\end{equation} 
While \eqref{xxxxx} changes to
\begin{equation}
    \sup_{0<|s|<C\epsilon_{j-1}^{-\frac{1}{16}}}\lVert M_s(\alpha,\theta,E,0)\rVert_0 \leq C \lVert\Phi_E\rVert_{0}^2\lVert B_j\rVert_0^2\leq C\Gamma^2 \epsilon_{j-1}^{-\frac{1}{96}},
\end{equation} 
where $ C $ is a universal constant. Then replace Lemma \ref{lem3.11} by the following lemma, which is contained in \cite{wang2021exact}.

\begin{lemma}[\cite{wang2021exact}]
    For any $ \delta_0 $ which is small enough, if $ E\in \mathcal{S}(\delta_0) $, then for sufficiently small $ \epsilon>0 $, $ N(E+\epsilon)-N(E-\epsilon)\geq c(\delta_0)\epsilon^{\frac{3}{2}}, $ where $ c(\delta_0) > 0 $ is a small universal constant.
\end{lemma}

\noindent Letting $ \delta_0>0 $ be arbitrarily small, we can show that if $ \alpha\in DC $, the restriction of the canonical spectral measure of $ \widetilde{H}_{v,\alpha,\theta} $ to $\mathcal{AR}$ is purely absolutely continuous, following the same line of reasoning as in the proof of Theorem~\ref{acthm}. Combining Theorem~\ref{ART} with $ \Sigma^{sub}\subset\mathcal{AR} $, we prove $ (2) $ of Theorem~\ref{mainmain}.

\section{Eigenvalues in Gaps of the Essential Spectrum}\label{sec.6}

Next we prove that there are at most finitely many eigenvalues in each gap which falls into the almost reducible regime. By the gap-labeling theorem \cite{gaplabel}, for any spectral gap $G$, there exists a unique $k\in \Z^d$ such that $N_{v, \alpha}(E)\equiv\langle k, \alpha\rangle \mod \Z$ in $G$.  That is, all the spectral gaps can be labelled by integer vectors: we
denote by $G_k(v ) = (E^{-}_k , E^{+}_k )$ the gap with label $k\neq 0$. When $E^{-}_k=E^+_k$, we say the gap is \emph{collapsed}. We also set $E_0^{-}:= \inf \Sigma_{v,\alpha}$, $E_0^+ := \sup \Sigma_{v,\alpha}$, and 
$G_0(v ) := (-\infty, E_0^{-}) \cup(E_0^+,\infty).$

\subsection{Proof of  Theorem \ref{acthm} (2):}

\subsubsection{Unbounded Gaps}\label{unbounded}

We first consider the two unbounded gaps comprising $G_0(v)$. The idea is the same as in the periodic case \cite{Tes1996Oscillation,Tes1999Jacobi}, that is, the goal is to show that one can find a solution $ u $ with 
\[
       u_n \text{ (at $ E_0^+ $) } \text{ or } (-1)^n u_n \text{ (at $ E_0^- $) }>c'>0, n\in \mathbb{Z}.
\] 
The following is essentially contained in  Lemma 5.1 and Corollary  5.1 of \cite{LWZZ2021Traveling}:

\begin{proposition}[\cite{LWZZ2021Traveling}]\label{positivesol}
Suppose $ \alpha\in DC_d(\gamma,\tau) $ and  $ v\in C^\omega(\T^d,\R) $ is sufficiently small. If
    \[
        \rho(\alpha,S_E^v) =0\text{ or }1/2,
    \] 
    then $ (\alpha,S_E^v) $ is reducible to $ \begin{pmatrix}
       \pm 1&c\\0&\pm 1
    \end{pmatrix} $, and  
    $ H_{v,\alpha,\theta}=Eu $ has a quasi-periodic or anti-quasi-periodic solution $u_n$. Specifically,
    \begin{enumerate}
    \item If  $ \rho(\alpha,S_E^v)=0$, then $u_n=u(\theta+n\alpha)>c'>0$.
    \item  If  $ \rho(\alpha,S_E^v)=1/2$, then $u_n=(-1)^nu(\theta+n\alpha)$ with $ (-1)^n u_n>c'>0$.
    \end{enumerate}
  \end{proposition}
  
\begin{remark}
    In fact, \cite{LWZZ2021Traveling} only considered the case $ \rho(\alpha,S_E^v)=0 $, but exactly the same proof works in the case $ \rho(\alpha,S_E^v)=1/2 $, the only difference is that the cocycle is reducible to  $ \begin{pmatrix}
        -1&c\\0&-1
    \end{pmatrix} $.
\end{remark}
By the fact $ N_{v,\alpha}(E)=1-2\rho_{v,\alpha}(E) $, we can deduce that 
\[
    \rho(\alpha,S_{E_0^-}^v)=1/2 , \qquad \rho(\alpha,S_{E_0^+}^v)=0.
\] 
Then by Proposition \ref{positivesol}, for $ E = E_{0}^+ $,  $ H_{v,\alpha,\theta}u=Eu $ has a positive quasi-periodic solution
\[
    u_n(E_0^+,\theta) =u(\theta+n \alpha)>c'>0.
\]

Next we need  the following simple observations:
\begin{corollary}\label{closesolution}
    Suppose $(\alpha, S_E^v )$ is reducible to a parabolic matrix, then $ H_{v,\alpha,\theta}u=Eu $ has a quasi-periodic solution $ u_n=u(\theta+n\alpha) $ or an anti-quasi-periodic solution $ u_n=(-1)^{n}u(\theta+n\alpha) $.   Furthermore, for any $ \epsilon>0 $, one can find $ \tilde{u}^{\pm} $ satisfying $ \widetilde{H}_{v,\alpha,\theta} \tilde{u}^{\pm} =E \tilde{u}^{\pm}   $ such that
    \[
    \limsup_{n\to\pm\infty}\|\tilde{u}^{\pm}_n-u_n\|\leq \epsilon.
    \]  
\end{corollary}

\begin{proof}
    Suppose $ (\alpha,S_E^v) $ is reducible to a parabolic matrix. Then we have 
\begin{equation}\label{redu1}
        B(\theta+\alpha)^{-1}S_E^v(\theta)B(\theta)=\begin{pmatrix}
            \pm 1&c\\
            0&\pm 1
        \end{pmatrix} = :C.
\end{equation}
    Without loss of generality, we consider the case $ \mathrm{tr} \, C=2 $.  Write
\[
    B(\theta)=\begin{pmatrix}
        b_{11}(\theta)&b_{12}(\theta)\\
        b_{21}(\theta)&b_{22}(\theta)
    \end{pmatrix}.
\] 
Then \eqref{redu1} implies that
\begin{eqnarray*}
(E-v(\theta))b_{11}(\theta)-b_{21}(\theta)&=&b_{11}(\theta+\alpha),\\
b_{11}(\theta)&=&b_{21}(\theta+\alpha),
\end{eqnarray*}
that is $$b_{11}(\theta+\alpha)+b_{11}(\theta-\alpha)+v(\theta)b_{11}(\theta)=Eb_{11}(\theta).$$
Hence $u_n=b_{11}(n\alpha+\theta)=b_{21}((n+1)\alpha+\theta)$ is a quasi-periodic solution of $ H_{v,\alpha,\theta}u=Eu $.

  By Lemma \ref{lem5.6}, for any $ \epsilon>0 $, one can find a solution $ \tilde{u}^+ $ of $ \widetilde{H}_{v,\alpha,\theta}u=Eu $, such that 
\[
    \limsup_{n\to\infty} \Big\| B(\theta+n\alpha)^{-1}\begin{pmatrix}
        \tilde{u}^+_n\\
        \tilde{u}^+_{n-1}
    \end{pmatrix}-\begin{pmatrix}
        1\\0
    \end{pmatrix} \Big\|\leq \frac{\epsilon}{\|B\|_0}.
\] 
Hence we have 
\[
    \limsup_{n\to\infty}\|\tilde{u}^{+}_n-u_n\|\leq \epsilon.
\]  
The proof in the case  $ \mathrm{tr} \, C=-2 $ is similar. 
\end{proof}

Then  means for 
   $ \epsilon$ sufficiently small, by Corollary \ref{closesolution}, there exists $ \tilde{u}^{\pm} (E_{0}^+,\theta) $ with $$ \widetilde{H}_{v,\alpha,\theta} \tilde{u}^{\pm}(E_0^+,\theta) =E_0^+ \tilde{u}^{\pm} (E_{0}^+,\theta)   $$ 
 such that 
\[
    \limsup_{n\to \pm\infty}\|\tilde{u}_n^{\pm}(E_{0}^+,\theta) - u_n(E_{0}^+,\theta)\|<\epsilon,
\] 
Combining this with Theorems~\ref{Teschl1} and \ref{Teschl3}, it follows that $ \sharp(\tilde{u}^{\pm} (E_{0}^+,\theta))<\infty $, and thus $ \dim \mathrm{Ran}P_{(E_{0}^+,\infty)}(\widetilde{H}_{v,\alpha,\theta})<\infty $. 
Similarly,  one can show that $ \dim \mathrm{Ran}P_{(-\infty,E_{0}^-)}(\widetilde{H}_{v,\alpha,\theta})<\infty $. In other words, $\widetilde{H}_{v,\alpha,\theta}$ has at most finitely many eigenvalues in $G_0(v)$ for any $\theta$.

\subsubsection{Bounded gaps}\label{bounded}

   For $ 2\tilde{\rho}= \langle k,\alpha\rangle\mod \Z $, $ k\neq 0 $, the gap-labeling theorem shows that $ \mathcal{R}_{\tilde{\rho}}=[E_k^-,E_k^+] $ if it is a non-collapsed gap, in this case, we recall the following result:

\begin{proposition}\label{conDep}
Let $ (\alpha,A_Ee^{f(E,\theta)})\in \T^d\times C^\omega(\T^d,\mathrm{SL}(2,\R)) $ be quasi-periodic cocycles continuous in $ E $. Assume that $\alpha \in \mathrm{DC}_d(\gamma,\tau)$ and $ f(E,\cdot)\in C^\omega_r(\T^d,\R)$ is sufficiently small.  For any $ \tilde \rho>0 $, we define
\[
    \mathcal{R}_{\tilde \rho}=\{E\in\R: \rho(\alpha, A_E e^{f(E,\theta)}) = \tilde \rho\}.
\] 
    If $ 2\tilde{\rho}= \langle k,\alpha\rangle\mod \Z $, then
    there exist $B(E,\cdot)\in C^{0}(\mathcal{R}_{\tilde{\rho}}\times 2\T^{d},\mathrm{SL}(2,\R))$, $C(E) \in C^0(\mathcal{R}_{\tilde{\rho}}, \mathrm{SL}(2,\R))$  such that
    $$
    B^{-1}(E,\theta+\alpha) A_E e^{f(E,\theta)} B(E,\theta)=C(E).
    $$
    Moreover, we have $ \rho(\alpha,C(E))=0 $.
    \end{proposition}
    \begin{remark}
        This result was first stated for Szeg\"o cocycles {\rm (}$\mathrm{SU}(1,1)$ case{\rm )} in  \cite[Proposition~5.3]{LDZ2022Cantor}, but the proof works for Schr\"odinger cocycles {\rm (}$\mathrm{SL}(2,\R)$ case{\rm )} with only minimal changes. 
    \end{remark}

    Thus, for any $ E\in[E_k^-,E_k^+] $, $ \rho(\alpha,C(E))=0 $. Without loss of generality we assume that $ \mathrm{tr} \, C(E)\geq 2 $. Then, $\mathrm{tr} \, C(E) = 2$ if and only if  $E= E_k^\pm $. Then one can  further assume that $ C(E) $ can be chosen as follows, $  $
    \[
        C(E)=\begin{pmatrix}
            \lambda(E)&\mu(E)\\
            0&\lambda(E)^{-1}
        \end{pmatrix},
    \] 
    where $ \lambda(E)\geq 1 $  for $ E\in [E_k^-,E_k^+] $ and $ \lambda(E) =1 $ if and only if $E= E_{k}^\pm $.
    To see this,    we can rewrite 
        \[
            C(E)=\exp M^{-1}\begin{pmatrix}
                \ii t(E)&e^{\ii 2\theta(E)}\nu(E)\\
                e^{-\ii 2\theta(E)}\nu(E)& -\ii t(E)
            \end{pmatrix}M,
        \] 
        where $$
M:=\frac{1}{1+\ii}\begin{pmatrix}1 & -\ii \\ 1 & \ii \end{pmatrix},
$$
        with $ t(E) $, $ \nu(E)\geq 0 $, $ \theta(E)\in C^0([E_k^-,E_k^+] ,\R) $, and $ \mathrm{tr} \, C(E)\geq 2 $ implies $ \nu(E)\geq |t(E)| $. If $ \nu(E)=0 $, then $ t(E)=0 $, and we are at one of the edge points $E^{\pm}_k$. 
        For $ E\in (E_k^-,E_k^+)  $, we can assume $ \nu(E)>0 $. 
        Solve 
        \begin{equation} \label{111}
            \nu(E)\sin 2(\theta(E)-\phi(E))=-t(E).
        \end{equation} 
        The solution $ \phi(E) $ is obviously continuous in $(E_k^-,E_k^+)$. Since if $ \nu(E)=0 $, we can solve \eqref{111} by any $ \phi(E) $, thus we can let $ \phi(E^{\pm}_k )=\lim_{E\to E^{\pm}_k}\phi(E) $. So we get $ \phi(E)\in C^0([E_k^-,E_k^+] ,\R) $ as desired. 
        Then we have 
        \[
            \begin{aligned}
                R_{-\phi(E)}C(E)R_{\phi(E)}&=\exp M^{-1}\begin{pmatrix}
                \ii t &e^{\ii (2\theta-2\phi)}\nu\\
                    e^{-\ii(2\theta-2\phi)}\nu&-\ii t
                \end{pmatrix}(E)M\\
                &=\exp \begin{pmatrix}
                    \cos(2\theta-2\phi)\nu&t-\sin(2\theta-2\phi)\nu\\
                    -t-\sin(2\theta-2\phi)\nu&-\cos(2\theta-2\phi)\nu
                \end{pmatrix}(E)\\
                &=\begin{pmatrix}
                    e^{\cos(2\theta-2\phi)\nu}&\mu\\
                    0&e^{-\cos(2\theta-2\phi)\nu}
                \end{pmatrix}(E).
            \end{aligned}
        \]

Notice that
    \[
    \left(\begin{matrix}
        u_{n}(E,\theta)\\u_{n-1}(E,\theta)
    \end{matrix}\right)=B(E;\theta+n\alpha)\begin{pmatrix}\lambda(E)^n&*\\0&\lambda(E)^{-n}\end{pmatrix} B(E;\theta)^{-1}\left(\begin{matrix}
        u_{0}(E,\theta)\\u_{-1}(E,\theta)
    \end{matrix}\right).
\] 
One can choose  suitable initial data $ \left(\begin{matrix}
    u_{0}(E,\theta)\\u_{-1}(E,\theta)
\end{matrix}\right) $ such that $$ B(E;\theta)^{-1}\left(\begin{matrix}
    u_{0}(E,\theta)\\u_{-1}(E,\theta)
\end{matrix}\right)=\begin{pmatrix}
    1\\0
\end{pmatrix} ,$$ and thus 
\[
    \left(\begin{matrix}
               u_{n}(E,\theta)\\       u_{n-1}(E,\theta)
    \end{matrix}\right)=\begin{pmatrix}
        \lambda(E)^n b_{11}(E;\theta+n\alpha)\\\lambda(E)^n b_{21}(E;\theta+n\alpha)
    \end{pmatrix}.
\] 
It follows that
\[
           u_{n}(E,\theta)=\lambda(E)^n b_{11}(E;\theta+n \alpha)=\lambda(E)^{n+1}b_{21}(E;\theta+(n+1)\alpha)
\] 
is a solution of $ H_{v,\alpha,\theta}u^E=E u^E $. Notice that for any $ E\in(E_k^-,E_k^+) $, $  u_{n}(E,\theta)$ is square-summable at $ -\infty $, thus 
$$ 
\{\lambda(E)^n b_{11}(E;\theta+n \alpha) \}_{n\in\Z}= \{u_{n}(E,\theta)\}_{n\in\Z}= \underline{u}(E,\theta)
$$
 is  the Weyl solution  at $ -\infty $.  Note  the Weyl solution $ \underline{u}(E_k^\pm,\theta) = \lim_{E \to E_k^\pm } \underline{u}(E,\theta)$ (see \cite[Sec 2.2]{Tes1999Jacobi} for details), then by  the continuity with respect to $E$ (Proposition \ref{conDep}), 
$$  \{b_{11}(E_k^\pm;\theta+n \alpha)\}_{n\in\Z}=  \underline{u}(E_k^\pm,\theta) .$$

    If we are at a collapsed gap, we do not need to do anything, so we assume $ \mu(E_k^{\pm})\neq 0 $. To complete the proof, it suffices to show for every gap edge, that it cannot be an accumulation point of discrete eigenvalues. Without loss of generality, the gap edge in question is $ E_k^- $ and we have $ [E_k^-, E_k^-+\delta] \subset [E_k^-, E_k^+]$ with $\delta>0$ chosen sufficiently small. A direct computation shows that
\begin{align*}
W(\underline{u}(E_k^-,\theta),\underline{u}(E_k^-+\delta,\theta))(n) & = \lambda(E_k^-+\delta)^n b_{11}(E_k^-+\delta;\theta+n \alpha) b_{11}(E_k^-;\theta+(n+1) \alpha) \\
    & \quad -\lambda(E_k^-+\delta)^{n+1} b_{11}(E_k^-+\delta;\theta+(n+1) \alpha)b_{11}(E_k^-;\theta+n \alpha).
\end{align*}


 Note that the resolvent set of $H_{v,\alpha,\theta}$ is independent of $\theta$, thus by Theorem \ref{Teschl4}, for all $ \theta\in\T^d $,  we have 
\[
    \sharp   W(\underline{u}(E_k^-,\theta),\underline{u}(E_k^-+\delta,\theta)) =\dim\mathrm{Ran} P_{(E_k^-,E_k^-+\delta)}(H_{v,\alpha,\theta})=0.
\] 
This especially means that $W(\underline{u}(E_k^-,\theta),\underline{u}(E_k^-+\delta,\theta)) (0) $  is non-vanishing for all $ \theta\in\T^d $.
On the other hand, if we denote 
\begin{eqnarray*}
    w(\theta)&=&W(\underline{u}(E_k^-,\theta),\underline{u}(E_k^-+\delta,\theta)) (0) \\
   & =&b_{11}(E_k^-+\delta;\theta) b_{11}(E_k^-;\theta+ \alpha)-\lambda(E_k^-+\delta) b_{11}(E_k^-+\delta;\theta+\alpha)b_{11}(E_k^-;\theta),
\end{eqnarray*} 
again by Proposition~\ref{conDep}, $  w(\theta) \in C^0(2\T^d,\R)$, and thus by continuity we have
\[
   w(\theta)>c'>0  \text{ or }  -  w(\theta) >c'>0.
\] 

By Corollary \ref{closesolution}, for any $\epsilon>0$ small enough,  there exist $ \tilde{u}^{\pm}({E_{k}^{-}},\theta) $ such that 
\begin{equation}\label{es1}
\limsup_{n\to \pm\infty}\|\tilde{u}_n^{\pm}({E_{k}^{-}},\theta) - b_{11}(E_k^-;\theta+n\alpha)\|<\epsilon,
\end{equation}
Meanwhile, while $\underline{u}(E,\theta) $ is  a solution of $ H_{v,\alpha,\theta}u^E=E u^E $,  we can obtain an approximate solution of  $ \widetilde{H}_{v,\alpha,\theta}$  with the help of the following lemma, which can seen as a counterpart of Lemma \ref{lem5.6}:

    \begin{lemma}\label{lem6.1}
 Let $|\lambda|>1$ and suppose that 
        \begin{equation} \label{6.3}
            \vec{\varphi} (n+1)=\left(A+R(n)\right)\vec{\varphi}(n),
        \end{equation}  
        where
        \[
            A=\begin{pmatrix}
                \lambda &c\\
                0&\lambda^{-1}
            \end{pmatrix}, \qquad \sum |R(n)|<\infty.
        \]
       Then for any $ \epsilon>0 $, we can find a solution $ \vec{\phi}_{\pm}(n) $ such that 
       \[
        \limsup_{n\to\pm\infty}|\lambda|^{-n} \Big\|\vec{\phi}_{\pm}(n)-\begin{pmatrix}
            \lambda^n\\0
        \end{pmatrix} \Big\| \leq \epsilon.
       \]  
    \end{lemma}
    \begin{proof}
       Denote by $ \Phi(n) $ the  fundamental matrix of $ \vec{\varphi}(n+1)=A\vec{\varphi}(n) $ and decompose
        \[
            \begin{aligned}
                \Phi(n)&=\begin{pmatrix}
                    \lambda^n&c\frac{\lambda^n-\lambda^{-n}}{\lambda-\lambda^{-1}}\\
                    0&\lambda^{-n}
                \end{pmatrix}=\begin{pmatrix}
                    0&-c\frac{\lambda^{-n}}{\lambda-\lambda^{-1}}\\
                    0&\lambda^{-n}
                \end{pmatrix}+\begin{pmatrix}
                    \lambda^n&c\frac{\lambda^n}{\lambda-\lambda^{-1}}\\
                    0&0
                \end{pmatrix}:=\Phi_1(n)+\Phi_2(n).
            \end{aligned}
        \] 
        It follows that
        \[
            \begin{aligned}
                \|\Phi_1(n)\|&\leq K |\lambda|^{-n},\  n\geq 0,\\
            \|\Phi_2(n)\|&\leq K |\lambda|^{n},\  n\leq 0.
            \end{aligned}
        \] 
        Let $ \vec{\psi}_0(n)=\Phi(n)\begin{pmatrix}
            1\\0
        \end{pmatrix} = \begin{pmatrix}
           \lambda^n\\0
        \end{pmatrix}  $, $ \ n\geq 1 $. One has
        \[
            \|\vec{\psi}_0(n)\|\leq |\lambda|^n.
        \]

      Choose  $ n_0 $ which is large  enough such that
        \[
            2K\sum_{s=n_0}^{\infty}|R(s)|<\epsilon,
        \] 
and define the sequence
        \[
            \vec{\psi}_{i+1}(n)=\begin{pmatrix}\lambda^n\\0\end{pmatrix} +\sum_{s=n_0+1}^n\Phi_1(n-s)R(s-1)\vec{\psi}_{i}(s-1)-\sum_{s=n+1}^\infty\Phi_2(n-s)R(s-1)\vec{\psi}_{i}(s-1).
        \] 
A direct computation shows that
        \[
            \begin{aligned}
                \|\vec{\psi}_{i+1}(n)-\vec{\psi}_{i}(n)\|&\leq \frac{|\lambda|^n}{2^{i+1}},
            \end{aligned}
        \]
        which means that there exists a limit function $ \vec{\psi}(n) $ on  $ n\geq n_0 $ that satisfies
        \[
            \|\vec{\psi}(n)\|\leq 2 |\lambda|^n,
        \] 
        \begin{equation}\label{iter2}
            \vec{\psi}(n)=\begin{pmatrix}\lambda^n\\0\end{pmatrix}+\sum_{s=n_0+1}^n\Phi_1(n-s)R(s-1)\vec{\psi}(s-1)-\sum_{s=n+1}^\infty\Phi_2(n-s)R(s-1)\vec{\psi}(s-1).
        \end{equation}
        It is easy to verify that  $ \vec{\psi}(n)$ is a solution of \eqref{6.3}. Next we give an estimate for $ \vec{\psi}(n)$. 
       
  By \eqref{iter2}, we have
        \[
            \begin{aligned}
                |\lambda|^{-n} \Big\| \vec{\psi}(n)-\begin{pmatrix}\lambda^n\\0\end{pmatrix} \Big\| & \leq 2K\sum_{s=n_0}^{n-1}|\lambda|^{-(n-s)}|R(s)||\lambda|^{s-n}+2K\sum_{s=n}^\infty |\lambda|^{n-s}|R(s)||\lambda|^{s-n}\\
                & \leq 2K\sum_{s=n_0}^{\infty}|R(s)|\leq \epsilon.
            \end{aligned}
        \] 
        The proof of the statement on the left half line (resp. $n\leq 0$) is similar.
    \end{proof}
    
 Once we have this, similarly as in Corollary~\ref{closesolution}, there exist $ \tilde{u}^{\pm}(E_{k}^{-}+\delta,\theta) $ such that 
\begin{equation}\label{es2}
\limsup_{n\to \pm\infty}\|\lambda(E_k^-+\delta)^{-n}\tilde{u}_n^{\pm}  (E_{k}^{-}+\delta,\theta) -  b_{11}(E_k^-+\delta;\theta+n\alpha)\|<\epsilon.
\end{equation}
By \eqref{es1} and \eqref{es2}, we have
\[
    \limsup_{n\to\pm\infty}\|\lambda(E_k^-+\delta)^{-n}W(\tilde{u}^{\pm}(E_{k}^{-},\theta),\tilde{u}^{\pm}(E_{k}^{-}+\delta,\theta))(n)-w(\theta+n\alpha)\|<\epsilon.
\]
Combining Theorem~\ref{Teschl2} and Theorem~\ref{Teschl3}, it follows that
\[
    \begin{aligned}
    \sharp W(\tilde{u}^{\pm}(E_{k}^{-},\theta),\tilde{u}^{\pm}(E_{k}^{-}+\delta,\theta))<\infty.
    \end{aligned}
\] 
This is equivalent to $ \dim \mathrm{Ran}P_{(E_k^-,E_k^-+\delta)}(H)<\infty $, which completes the proof. \qed

\subsection{Proof of Theorem \ref{mainmain} (3):}

In the unbounded gap $G_0(v)$ case, without loss of generality, assume $E_0^{-}\in\Sigma^{sub}$. Then, $(\alpha, S_{E_0^-}^v)$ is almost reducible by Theorem~\ref{ART}, consequently, it  is reducible to $ \begin{pmatrix}
\pm 1&c\\0&\pm 1 \end{pmatrix} $ \cite{AJ2010Almost,eliasson}. However, it is not known whether $ H_{v,\alpha,\theta}=Eu $ has a uniformly positive quasi-periodic solution. The proof we adopt will be different from Section \ref{unbounded}, and the same as in the bounded gap case (Section \ref{bounded}). 

Now we look at the bounded gap case,  consider $ [E_k^-,E_k^+] $ with $ E_k^- $ or $ E_k^+\in\Sigma^{sub} $. Without loss of generality, assume it is $ E_k^- $. There exist $ \bar{h}>0 $, $ \Phi_{E_k^-}\in C^\omega(2\T,\mathrm{SL}(2,\R)) $
 such that for any $ E\in [E_k^-,E_k^-+\delta] $, $ \delta $ sufficiently small:
$$
        \Phi_{E_k^-}(\theta+\alpha)^{-1}S_E^v(\theta)\Phi_{E_k^-}(\theta)=R_{E_k^-}e^{f(E,\theta)}
$$ 
with $ \|f(E,\theta)\|_{\bar{h}} $ small enough. Obviously, $ R_{E_k^-}e^{f(E,\theta)}\in C^0[E_k^-,E_k^-+\delta] $. 
Notice that $ 2\rho(\alpha,R_{E_k^-}e^{f(E,\theta)})=\langle k+\deg\Phi_{E_k^-},\alpha\rangle $.  Then Proposition \ref{conDep} applies, there exists $ \bar{B}(E,\cdot)\in C^0([E_k^-,E_k^-+\delta]\times 2\T,\mathrm{SL}(2,\R)) $, $ C(E) \in C^0([E_k^-,E_k^-+\delta], \mathrm{SL}(2,\R))$ such that 
\[
    \bar{B}(\theta+\alpha)^{-1}\Phi_{E_k^-}(\theta+\alpha)^{-1}S_E^v(\theta)\Phi_{E_k^-}(\theta)\bar{B}(\theta)=C(E),
\]  
with $ \rho(\alpha,C(E))=0 $. The rest proof is the same as in Section~\ref{bounded}, we omit the details.

\section{Completion of the Proof of Theorem \ref{mainmain}}\label{sec.7}

In this section we prove parts  $(1)$ and $(4)$ of Theorem \ref{mainmain}.  By Avila's global theory  \cite{avila2015global}, for typical $ v\in C^\omega(\T,\R) $, we have $ L(\alpha,E)>c_0>0 $ for all $ E $ in any compact interval of $ \R\backslash\Sigma^{sub} $, thus we get part $ (4) $ of Theorem \ref{mainmain} by applying Theorem \ref{mainresult}. By Weyl's criterion, $ \sigma_\mathrm{ess}(\widetilde{H}_{\lambda,\alpha,\theta})=\sigma_\mathrm{ess}(H_{\lambda,\alpha,\theta}) $. As we have shown that there is no singular continuous spectrum in $ \Sigma=\sigma_\mathrm{ess}(H_{\lambda,\alpha,\theta}) $, this completes the proof of $(1)$ of Theorem \ref{mainmain}.

\section*{Acknowledgements} 
D.Damanik was supported in part by NSF grants DMS--1700131 and DMS--2054752, an Alexander von Humboldt Foundation research award, and Simons Fellowship $\# 669836$. J.You and Q. Zhou were  partially supported by National Key R\&D Program of China (2020 YFA0713300) and Nankai Zhide Foundation.  J. You was also partially supported by NSFC grant (11871286). Q. Zhou was supported by NSFC grant (12071232), the Science Fund for Distinguished Young Scholars of Tianjin (No. 19JCJQJC61300).

\begin{appendix}

\section{Eigenvalues at Gap Edges}

We discuss the occurrence of eigenvalues at gap edges for decaying potentials. This observation, which we learned from Milivoje Lukic, is possibly well known. In any event, we spell it out explicitly here for the convenience of the reader.

A common trick to force certain eigenfunction behavior is to choose an appropriate proposed eigenfunction along with an energy, and to then deduce the (asymptotic) form of the potential from it.

Let us choose the eigenfunction $u$, which obeys $u(n) = \frac{1}{n}$ for $|n| \ge n_0$ and is chosen for $|n| < n_0$ so that it obeys any desired boundary condition at the origin in the half-line case or in some arbitrary way in the whole-line case. The energy is $E = 2$. Since we need to satisfy
$$
u(n+1) + u(n-1) + V(n) u(n) = 2 u(n),
$$
it follows that for $|n| > n_0$, we must have
$$
V(n) = 2 - \frac{u(n+1) + u(n-1)}{u(n)} = - \frac{2}{n^2 - 1}.
$$

This shows the following:

\begin{proposition}
There exists a potential decaying at the rate $n^{-2}$ for which $2$ is an eigenvalue. In particular, an $\ell^1$ potential does not necessarily ensure the purity of the absolutely continuous spectrum on $[-2,2]$.
\end{proposition}

\end{appendix}
\bibliographystyle{spmpsci}
\bibliography{bourgainALdecay_referrence}
\end{document}